\documentclass[11pt]{amsart}

\usepackage[utf8]{inputenc} 			
\usepackage[T1]{fontenc} 			
\usepackage[francais, english]{babel} 			
\usepackage{layout}					
\usepackage{setspace}				
\usepackage{graphics}				
\usepackage{cite}

\usepackage{amsthm}
\usepackage{amsmath}
\usepackage{amssymb}
\usepackage{mathrsfs}
\usepackage{dsfont}			
\usepackage{enumerate}
\usepackage{enumitem}

\usepackage{mathtools,amscd}            
\usepackage{tikz-cd}            

\usepackage[pdftex,final,linktoc=all,hyperindex,breaklinks]{hyperref}
\hypersetup{
    linktoc=page,	
    linkcolor=blue,          
    citecolor=blue,        
    filecolor=blue,      
    urlcolor=blue,
   colorlinks=true      }   

\usepackage{geometry}

\newcommand\R{\mathbb{R}}		
\newcommand\G{\mathbb{G}}		
\newcommand\C{\mathbb{C}}
\newcommand\Q{\mathbb{Q}}

\newcommand\N{\mathbb{N}}

\newcommand\F{\mathbb{F}}

\newcommand\M{\mathbb{M}}

\newcommand\LL{\mathscr{L}}
\newcommand\LLr{\mathscr{L}_{\mathrm{ring}}}
\newcommand\LLf{\mathscr{L}_{\mathrm{field}}}

\newcommand\MM{\mathscr{M}}
\newcommand\NN{\mathscr{N}}
\newcommand\KK{\mathscr{K}}

\newcommand\CCCC{\mathscr{C}}

\newcommand\Z{\mathbb{Z}}

\newcommand{\set}[1]{\left\{ {#1} \right\}}
\newcommand{\ol}[1]{\overline{{#1}}}
\newcommand{\vect}[1]{\langle {#1} \rangle}
\newcommand{\abs}[1]{\lvert {#1} \rvert}

\def\Ind#1#2{#1\setbox0=\hbox{$#1x$}\kern\wd0\hbox to 0ex{\hss$#1\mid$\hss}
\lower.9\ht0\hbox to 0ex{\hss$#1\smile$\hss}\kern\wd0}
\def\Notind#1#2{#1\setbox0=\hbox{$#1x$}\kern\wd0\hbox to 0ex{\mathchardef\nn="0236\hss$#1\nn$\kern1.4\wd0\hss}\hbox to 0ex{\hss$#1\mid$\hss}\lower.9\ht0
\hbox to 0ex{\hss$#1\smile$\hss}\kern\wd0}

\def\ind{\mathop{\mathpalette\Ind{}}}

\def\nind{\mathop{\mathpalette\Notind{}}}

\def\indi#1{\mathop{\ \ \hbox to 0ex{\hss$\vert^{\hbox to 0ex{$\scriptstyle#1$\hss}}$\hss}
\lower1ex\hbox to 0ex{\hss$\smile$\hss}\ \ }}

\def\nindi#1{\mathop{\ \ \hbox to 0ex{\hss$\!\not{\vert}^{\hbox to 0ex{$\scriptstyle\,#1$\hss}}$\hss}
\lower1ex\hbox to 0ex{\hss$\smile$\hss}\ \ }}

\newcommand{\findep}[1][]{%
  \mathrel{
    \mathop{
      \vcenter{
        \hbox{\oalign{\noalign{\kern-.3ex}\hfil$\vert$\hfil\cr
              \noalign{\kern-.7ex}
              $\smile$\cr\noalign{\kern-.3ex}}}
      }
    }\displaylimits_{#1}
  }
}

\newcommand{\nfindep}[1][]{%
  \mathrel{
    \mathop{
      \vcenter{
	\hbox{\oalign{\noalign{\kern-.3ex}\hfil$\!\not{\vert}$\hfil\cr
              \noalign{\kern-.7ex}
              $\smile$\cr\noalign{\kern-.3ex}}}
      }
    }\displaylimits_{#1}
  }
}

\makeatletter
\newcommand{\setword}[2]{%
  \phantomsection
  #1\def\@currentlabel{\unexpanded{#1}}\label{#2}%
}
\makeatother

\newcommand{\NSOP}[1]{\mathrm{NSOP}_{#1}}

\newcommand{\ACFG}{\mathrm{ACFG}}

\newcommand{\ACF}{\mathrm{ACF}}
\newcommand{\RCF}{\mathrm{RCF}}
\newcommand{\PAC}{\mathrm{PAC}}

\newcommand{\Psf}{\mathrm{Psf}}
\newcommand{\DCF}{\mathrm{DCF}}
\newcommand{\SCF}{\mathrm{SCF}}

\newcommand{\ACFA}{\mathrm{ACFA}}
\newcommand{\acl}{\mathrm{acl}}

\newcommand{\di}{\mathrm{dim}}

\theoremstyle{definition}
\newtheorem{df}{Definition}[section]

\newtheorem{ex}[df]{Example}
\theoremstyle{theorem}
\newtheorem{prop}[df]{Proposition}		
\newtheorem{thm}[df]{Theorem}
\newtheorem*{thm*}{Theorem}
\newtheorem{lm}[df]{Lemma} 
\newtheorem{cor}[df]{Corollary} 
\newtheorem{fact}[df]{Fact}

\theoremstyle{remark}
\newtheorem*{claim}{Claim}
\usepackage{etoolbox}
\AtEndEnvironment{proof}{\setcounter{claimn}{0}}
\newtheorem{rk}[df]{Remark}

\newsavebox{\auteurbm}

\let\oldabstract\abstract
\let\oldendabstract\endabstract
\makeatletter
\renewenvironment{abstract}
{%
               {\list{}{\addtolength{\leftmargin}{8em} 
                        \listparindent 1.5em%
                        \itemindent    \listparindent%
                        \rightmargin   \leftmargin%
                        \parsep        \z@ \@plus\p@}%
                \item\relax}%
               {\endlist}%
\oldabstract}
{\oldendabstract}
\makeatother

\title{Generic Expansions by a Reduct}

\author{Christian d'Elb\'ee}\thanks{Intitut Camille Jordan, Université Lyon 1.}
\date{\small\today}

\begin{document}

\maketitle				

\begin{abstract}
  \noindent Consider the expansion $T_S$ of a theory $T$ by a predicate for a submodel of a reduct $T_0$ of $T$. We present a setup in which this expansion admits a model companion $TS$. We show that the nice features of the theory $T$ transfer to $TS$. In particular, we study conditions for which this expansion preserves the $\NSOP{1}$-ness, the simplicity or the stability of the starting theory $T$. We give concrete examples of new $\NSOP{1}$ not simple theories obtained by this process, among them the expansion of a perfect $\omega$-free $\PAC$ field of positive characteristic by generic additive subgroups, and the expansion of an algebraically closed field of \emph{any} characteristic by a generic multiplicative subgroup.
\end{abstract}

\hrulefill
\tableofcontents	
\hrulefill

\section*{Introduction}

 Existentially closed models of a theory have in general some randomness --- or \emph{generic} --- aspect, resulting from their definition, that allows a reasonable description of their algebra of definable sets. Informally, we will call \emph{generic} a theory (or a model of such theory) that axiomatises structures that are existentially closed in a reasonable class of extension.\\

 \noindent Let $T$ be a theory in a language $\LL$. Let $T_0$ be a reduct of $T$. Let $\LL_S = \LL\cup\set{S}$, for $S$ a new unary predicate symbol, and $T_S$ be the $\LL_S$-theory whose models $(\MM,\MM_0)$ consist in a model $\MM$ of $T$ in which $S$ is a predicate for a model $\MM_0$ of $T_0$ which is a substructure of $\MM$. In this paper we present a setting for an axiomatisation of generic models of $T_S$, this axiomatisation is denoted by $TS$, it is most of the time the model-companion of the theory $T_S$.\\

 \noindent This generic expansion produces numerous examples (Section~\ref{part_ex}) of new theories that are, in general, not simple, not even when $T$ is strongly minimal (see for instance the theory $\ACFG$ in Example~\ref{ex_ACFG_NSOP}). However, most of these new theories turn out to be $\NSOP 1$.\\

 \noindent $\NSOP 1$ theories, for ``not strong order property 1'', were defined by D\v{z}amonja and Shelah in~\cite{DS04} (together with $\NSOP 2$) as an extension of the $(\NSOP{n})_{n\geq 3}$ hierarchy. In~\cite{SU08} Shelah and Usvyatsov proved that $T_{feq}^*$ (the model completion of the theory of infinitely many independent parametrized equivalence relations) is $\NSOP{1}$ and not simple. For the past three years, $\NSOP{1}$ theories have been intensively studied.\\

\noindent A first breakthrough in the study of $\NSOP{1}$ theories was made by Chernikov and Ramsey in~\cite{CR16}. They proved a Kim-Pillay style result~\cite{CR16} which states that a theory is $\NSOP{1}$ provided there exists an independence relation satisfying some specific properties. This result turned out to be a very useful tool to prove that a theory is $\NSOP{1}$. The $\omega$-free $\PAC$ fields case is a good example. A $\PAC$ field is simple if~\cite{CP98} and only if~\cite{C99} it is bounded. Nonetheless, in her work~\cite{C02} on $\omega$-free PAC fields (which are unbounded), Chatzidakis defined a weak notion of independence and showed that it satisfied some nice properties, in particular, the so-called \emph{independence theorem}. It turned out that almost all the properties of the criterion~\cite{CR16} were proved at that time. Chernikov and Ramsey used this weak independence to deduce that the theory of $\omega$-free PAC fields is $\NSOP{1}$. They also showed that Granger's example of generic bilinear form over an infinite dimensional vector space over an algebraically closed field is $\NSOP{1}$ (see~\cite{G99} or~\cite[Example 6.1]{CR16}), as well as the combinatorial example of a generalised parametrized structure (see~\cite[Example 6.3]{CR16}).\\

\noindent A second breakthrough was the development of Kim-independence by Kaplan and Ramsey in~\cite{KR17}. They introduced analogues of forking and dividing --Kim-forking and Kim-dividing-- which behave nicely in $\NSOP{1}$ theories. Kim-dividing is defined as dividing with respect to some particular indiscernible sequences, namely sequences in a global invariant type. Numerous properties of forking in simple theories appear for Kim-forking in $\NSOP 1$ theories. For instance, a theory is $\NSOP{1}$ if and only if Kim-independence is symmetric. Kaplan and Ramsey also completed the Kim-Pillay style criterion in~\cite{CR16} to get a characterisation of Kim-independence in terms of properties of a ternary relation, similarly to the Kim-Pillay classical result. Using this tool, they identified Kim-independence in various $\NSOP{1}$ theories. Chatzidakis' weak independence in $\omega$-free PAC fields turned out to be Kim-independence.\\

\noindent In Sections~\ref{sec_ind}, starting from an independence relation $\indi T$ in $T$, we define independence relations in $TS$ and identify which properties of $\indi T$ are transferred to those new independence relations in $TS$, and under which conditions. This allows us to exhibit hypotheses under which the expansion from $T$ to $TS$ preserves $\NSOP 1$, simplicity or stability (Section~\ref{part_pres}). We also give a general description of Kim-independence in $TS$ in that context.\\

\noindent Finally in Section~\ref{part_ex}, we prove the existence of new generic theories and show that most of them are $\NSOP 1$. The existence is based on definability results in the theory $T$. For the expansion of a perfect $\omega$-free $\PAC_p$-field of positive characteristic by a generic additive subgroup, the elimination of $\exists^{\infty}$ is enough. However, for the expansion of an algebraically closed field (of any characteristic) by a generic multiplicative subgroup, it relies on the definability of the \emph{freeness}\footnote{
Let $W\subset K^n\setminus\set{(0,\cdots,0)}$ be an affine irreducible algebraic variety in an algebraically closed field $K$ of characteristic $p\geq 0$. We say that $W$ is \emph{free} if it is not contained in any translate of a proper algebraic subgroup of the torus $\G_m^n(K)$.} 
of a family of parametrized affine variety (Subsection~\ref{sec_genmult}). By contrast, if proving that the latter theory is $\NSOP 1$ is relatively straightforward using the results of Section~\ref{part_pres}, proving that the expansion of a perfect $\omega$-free $\PAC_p$-field of positive characteristic by a generic additive subgroup is $\NSOP 1$ is more difficult (Subsection~\ref{sec_NSOPf}).\\ 

\noindent We end Section~\ref{part_ex} with an example where the model-companion $TS$ does not exist: the expansion of a field of characteristic zero by a predicate for an additive subgroup. In the generic expansion of any field by an additive subgroup, the multiplicative stabiliser is always definable and of fixed cardinality, it is the prime ring, so either a finite field or $\Z$. In the multiplicative case (Subsection~\ref{sec_genmult}) we do not have this dichotomy on the characteristic since the “exponential” stabiliser is not definable.


 \clearpage

 \noindent\textbf{Conventions and notations.} Capital letters $A,B,C$ stands for sets whereas small latin letters $a,b,c$ designate either singletons, finite or infinite tuples. For a given theory $T$, we use standard model-theoretic notations, such as $tp^T(a/C)$, $a\equiv^T_C a'$, $\acl_T$, etc\dots

We often identify tuples and sets when dealing with independence relations, for some tuple $a=a_1,\dots$ then $c\ind_C a$ has the same meaning as $c\ind_C \set{a_1,\dots}$. 
Here is a list of properties for a ternary relation $\ind$ defined over small subsets of $\M$ a big model of some countable theory $T$.
\begin{itemize}
  \item \setword{\bsc{Invariance}}{INV}. If $ABC\equiv A'B'C'$ then $A\ind_C B$ if and only if $A'\ind_{C'} B'$.
  \item \setword{\bsc{Finite Character}}{FIN}. If $a\ind_C B$ for all finite $a\subseteq A$, then $A\ind_C B$.
  \item \setword{\bsc{Symmetry}}{SYM}. If $A\ind_C B$ then $B\ind_C A$.
  \item \setword{\bsc{Closure}}{CLO} $A\ind_C B$ if and only if $A \ind_{\acl(C)} \acl(BC)$.
  \item \setword{\bsc{Monotonicity}}{MON}. If $A\ind_C BD$ then $A\ind_C B$.
  \item \setword{\bsc{Base Monotonicity}}{BMON}. If $A\ind_C BD$ then $A\ind_{CD} B$.
  \item \setword{\bsc{Transitivity}}{TRA}. If $A \ind_{CB} D$ and $B\ind_C D$ then $AB\ind_C D$.
  \item \setword{\bsc{Existence}}{EX}. For any $C$ and $A$ we have $A\ind_{C} C$.
  \item \setword{\bsc{Full Existence}}{EXT}. For all $A,B$ and $C$ there exists $A'\equiv_C A$ such that $A'\ind_C B$.
  \item \setword{\bsc{Extension}}{EXT2}. If $A\ind_C B$, then for all $D$ there exists $A'\equiv_{CB}A$ and $A'\ind_C BD$.
  \item \setword{\bsc{Local Character}}{LOC}. For all finite tuple $a$ and infinite $B$ there exists $B_0\subset B$ with $\abs{B_0}\leq \aleph_0$ and $a\ind_{B_0} B$.
  \item \setword{\bsc{Strong Finite Character}}{STRFINC} over $E$. If $a\nind_E b$, then there is a formula $\Lambda(x,b,e)\in tp(a/Eb)$ such that for all $a'$, if $a'\models \Lambda(x,b,e)$ then $a'\nind_E b$.
\item $\ind '$-\setword{\bsc{amalgamation}}{AM} over $E$. If there exists tuples $c_1,c_2$ and sets $A,B$ such that
\begin{itemize}
\item $c_1\equiv_E c_2$
\item $A\ind'_E B$
\item $c_1\ind_E A$ and $c_2\ind_C B$
\end{itemize}
then there exists $c\ind_E A,B$ such that $c\equiv_A c_1$, $c \equiv_B c_2$, $A\indi a _{Ec} B$, $c\indi a _{EA} B$ and $c\indi a _{EB} A$.
\item \setword{\bsc{Stationnarity}}{STAT} over $E$. If $c_1 \equiv_E c_2$ and $c_1\ind_E A$, $c_2\ind_E A$ then $c_1\equiv_{EA} c_2$.
\item \setword{\bsc{Witnessing}}{WIT}. Let $a,b$ be tuples, $\MM$ a model and assume that $a\nind_{\MM} b$. Then there exists a formula $\Lambda(x,b)\in tp(a/\MM b)$ such that for any global extension $q(x)$ of $tp(b/\MM)$ finitely satisfiable in $\MM$ and for any $(b_i)_{i<\omega}$ such that for all $i<\omega$ we have  $b_i\models q\upharpoonright \MM b_{<i}$, the set
 $\set{ \Lambda(x, b_i) \mid i<\omega}$ is inconsistent.
\end{itemize}
If $A\ind _C B$, the set $C$ is called the \emph{base set}. For two ternary relations $\ind$ and $\ind'$, the notation $\ind \rightarrow \ind'$ means that for all $A,B,C$, if $A\ind _C B$ then $A\ind'_C B$. The independence relation $\indi a$ is defined by $A\indi a _C B \iff \acl_T(AC)\cap \acl_T(BC) = \acl_T(C)$, with respect to some theory $T$.

\section{The generic expansion by a predicate for a reduct}\label{part_gen_red}

Let $T$ be an $\LL$-theory. Let $\LL_0\subseteq \LL$ and let $T_0$ be a reduct of $T$ in the language $\LL_0$. Let $S$ be a new unary predicate symbol and set $\LL_S = \LL\cup\set{S}$. We denote by $T_S$ the $\LL_S$-theory of $\LL_S$-structures $(\MM,\MM_0)$ where $\MM \models T$ and $S(\MM) =\MM_0 \models T_0$ is  a substructure of $\MM\upharpoonright\LL_0$. We aim to describe a favorable context for the existence of a theory $TS$ that axiomatises generic models of $T_S$.

We denote by $\acl_0$ the algebraic closure in the sense of $T_0$. Assume that $T_0$ is pregeometric (i.e. $\acl_0$ satisfies exchange), there is an associated independence relation $\indi 0$ (see for instance $\indi{cl}$ in~\cite[C1]{TZ12}).
It is defined over every subset of any model of $T_0$ and satisfies the properties \ref{FIN}, \ref{SYM}, \ref{CLO}, \ref{MON}, \ref{BMON}, \ref{TRA}. In particular, $\indi 0$ is defined over every subset of any model of $T$, and we will only use it over small subsets of some monster model $\M$ of $T$. The property \ref{SYM} of $\indi 0$ will be tacitly used throughout this chapter.

\begin{df}
  Let $t$ be a single variable and $x,y$ two tuples of variables. We say that a formula $\psi(t,y)$ is \emph{$n$-algebraic in $t$} (or just \emph{algebraic in $t$}) if for all tuple $b$ the number of realisations of $\psi(t,b)$ is at most $n$. In that context we say that a formula $\psi(t,x,y)$ is \emph{strict in $y$} if whenever $b$ is an $\indi 0$-independent tuple over $a$, the set of realisations of $\psi(t,a,b)$ is in $\acl_0(a,b)\setminus \acl_0(a)$.
\end{df}

If $\psi(t,b)$ is an $\LL_0$-algebraic formula, there exists an $\LL_0$-formula $\tilde \psi(t,x)$ algebraic in $t$ such that $\psi(\MM,b)\subseteq \tilde \psi(\MM,b)$, for all $\MM\models T_0$.

\begin{ex}
In the language of vector spaces, the formula $t = \lambda x + \mu y$ is strict in $y$ if and only if $\mu\neq 0$.
\end{ex}

\begin{lm}\label{lm_strict}
  Assume that $T_0$ is pregeometric. Then for $u$ a singleton and tuples $a$ and $b$, if $u\in \acl_0(a,b)\setminus \acl_0(a)$, there exists an $\LL_0$-formula $\tau(t,x,y)$ algebraic in $t$ and strict in $y$ such that $u\models \tau(t,a,b)$.
\end{lm}

\begin{proof}
Assume that $b = b_1,\dots,b_n$. By hypothesis and using exchange, we may assume that $b_1\in \acl_0(u,a,b_2,\dots,b_n)$. Let $\tau_1(t,a,b)$ be an $\LL_0$-formula algebraic in $t$ isolating the type $tp^{T_0}(u/ab)$ and $\tau_2(y_1, u, a,b_2,\dots,b_n)$ algebraic in $y_1$ isolating $tp^{T_0}(b_1/u,a,b_2,\dots,b_n)$. Then $\tau(t,x,y) = \tau_1(t,x,y) \wedge \tau_2(y_1,t,x,y_2,\dots,y_n)$ is strict in $y$. Indeed assume that for some independent tuple $b'$ over $a'$, and singleton $u'$ we have $\models \tau(u',a',b')$. It follows that $u'\in \acl_0(a'b')$ and $b_1'\in \acl_0(u',a',b_2'\dots,b_n')$. If $u'\in \acl_0(a')$ then $b_1'\in \acl_0(a',b_2',\dots, b_n')$ contradicting that $b'$ is $\indi0$-independent over $a'$, so $u'\notin \acl_0 ( a')$.
\end{proof}

\begin{df}
An expansion $(\MM,\MM_0)\subseteq (\NN,\NN_0)$ is \emph{strong} if $\NN_0 \indi{0}_{\MM_0} \MM$.
\end{df}

\begin{thm}\label{model_com_gen}
Assume that the following holds:
\begin{enumerate}[label={$(H_{\arabic*})$}]
  \item\label{H1} $T$ is model complete;
  \item\label{H2} $T_0$ is model complete and for all infinite $A$, $\acl_0(A) \models T_0$;
  \item\label{H3} $T_0$ is pregeometric;
  \item\label{H4} for all $\LL$-formula $\phi(x,y)$ there exists an $\LL$-formula $\theta_\phi(y)$ such that for $b\in \MM\models T$,
\begin{eqnarray*}
\MM \models \theta_\phi(b) &\iff& \text{there exists $\NN\succ \MM$ and $a\in \NN$ such that}\\
 &\mbox{ }& \text{ $\phi(a,b)$ and $a$ is an $\indi 0$-independent tuple over $\MM$.}
\end{eqnarray*}
\end{enumerate}
Then there exists a theory $TS$ containing $T_S$ such that
\begin{itemize}
\item every model of $T_S$ has a strong extension which is a model of $TS$;
\item if $(\MM,\MM_0)\models TS$ and $(\NN, \NN_0)\models T_S$  is a strong extension of $(\MM,\MM_0)$ then $(\MM,\MM_0)$ is existentially closed in $(\NN,\NN_0)$.
\end{itemize}
An axiomatization of $TS$ is given by adding to $T_S$ the following axiom scheme: for each tuple of variables $x  = x^0x^1$, for $\LL$-formula $\phi(x,y)$, and $\LL_0$-formulae $(\tau_i(t,x,y))_{i<k}$ which are algebraic in $t$ and strict in $x^1$,
$$\forall y (\theta_\phi(y) \rightarrow (\exists x \phi(x,y)\wedge x^0\subseteq S\wedge \bigwedge_{i<k} \forall t \ (\tau_i(t,x,y)\rightarrow t\notin S))).$$
\end{thm}

\begin{proof}
We prove the first assertion.
Let $(\MM,\MM_0)$ be a model of $T_S$, $\phi(x,y)$ an $\LL$-formula and a partition $x = x^0x^1$. Assume that for some tuple $b$ from $\MM$ we have $\theta_\phi(b)$. We show that the conclusion of the axiom can be satisfied in a strong extension $(\NN,\NN_0)$ with $\NN\succ \MM$. Then the result will follow by taking the union of a chain of models of $T_S$, which is again a model of $T_S$ because it is an elementary chain of models of $T$ with a predicate for models of $T_0$ which is inductive, by model-completeness. The fact that the union of a chain of strong extensions is again strong follows from \ref{FIN} and \ref{TRA} of $\indi{0}$, and the model-completeness of $T_0$.

By \ref{H4} there exists an extension $\NN\succ \MM$, and a tuple $a\in \NN$ satisfying $\phi(x,b)$ and such that  $a$ is $\indi 0$-independent over $\MM$. Set $\NN_0 = \acl_0(\MM_0 a^0)$. Then using \ref{MON}, \ref{BMON} and \ref{CLO} of $\indi{0}$, $a^0\MM_0\indi{0}_{\MM_0} \MM$. This means that the extension $(\MM,\MM_0)\subseteq (\NN,\NN_0)$ is strong. Now clearly $a^0\subseteq S$. Using \ref{BMON} and \ref{CLO}, it follows that $ab\indi{0}_{a^0b} \MM_0a^0$. Take any $\LL_0$-formula $\tau(t,x,y)$ algebraic in $t$ and strict in $x^1$, and assume that $u\in \NN$ satisfies $\tau(t,a,b)$. As $\tau$ is strict in $x^1$ and $a^1$ is $\indi 0$ -independent over $b a^0$, we have $u\in \acl_0(ab)\setminus \acl_0(a^0b)$. If $u\in \NN_0$ then it belongs to $\acl_0(ab)\cap \acl_0(\MM_0a^0) \subseteq \acl_0(a^0b)$, a contradiction, hence $u\notin S$. It follows that $(\NN,\NN_0)\models \phi(a,b)\wedge a^{0}\subseteq S\wedge \bigwedge_{i<k} \forall t \ (\tau_i(t,a,b)\rightarrow t\notin S)))$.

We now prove the second assertion.

Let $(\MM,\MM_0)\models TS$ and $(\NN,\NN_0)\models T_S$, a strong extension of $(\MM,\MM_0)$. Take finite tuples $a\in \NN$ and $b\in \MM$.
To understand the quantifier-free $\LL_S$-type of $a$ over $b$, it is sufficient to deal with formulae of the form
$$\psi(x,b)\wedge \bigwedge_{i\in I} x_i\in S \wedge \bigwedge_{j\in J} x_j\notin S$$
with $\psi(x,y)$ an $\LL$-formula. The reduction to formulae of this form is done by increasing the length of $x$ (replacing $\LL$-terms by variables), which may be greater than $\abs{a}$. We assume that $a$ satisfies the formula above.

\begin{claim} There exists an $\indi 0$-independent tuple $a' = a^0{}'a^1{}'$ such that $\acl_0(\MM a) = \acl_0(\MM a')$ with
\begin{enumerate}
\item $a'\indi{0} \MM$;
\item $\acl_0(a')\cap \NN_0 = \acl_0(a^0{}')$;
\item $\NN_0\cap \acl_0(\MM,a') = \acl_0(\MM_0, a^0{}')$.
\end{enumerate}
\end{claim}
\begin{proof}[Proof of the claim] Take a tuple $a^0{}'$ in $\NN_0\cap \acl_0(\MM,a)$ maximal $\indi 0$-independent over $\MM_0$.
We have $a^0{}'\indi{0} \MM_0$, and as the extension is strong we also have $a^0{}'\indi{0} \MM$ by \ref{TRA}. Now take a tuple $a^1{}'$ in $\acl_0(\MM a)$ maximal $\indi 0$-independent over $\acl_0(\MM a^0{}')$. We have $a^1{}'\indi{0}\MM a^0{}'$ and so $a^0{}'a^1{}'\indi{0} \MM$. Set $a' = a^0{}'a^1{}'$ and the claim holds.
\end{proof}

Now as $a\subseteq \acl_0(\MM,a')$ there exists a finite tuple $m^1$ from $\MM$ $\indi 0$-independent over $\MM_0a'$ such that $a\subseteq \acl_0(\MM_0m^1 a')$. Similarly there exists a finite tuple $m^0$ from $\MM_0$ with $m^0\indi{0} m^1a'$ such that $a\subseteq \acl_0(m^0 m^1 a')$.

If $i\in I$, using \textit{(3)}, we have $a_i\in \acl_0(\MM_0a^0{}')\cap \acl_0(m^0m^1 a') = \acl_0(m^0a^0{}')$. Hence there is an $\LL_0$-formula $\tau_i(t, a^0{}',m^0)$ algebraic in $t$ such that $a_i \models \tau_i(t,a^0{}',m^0)$.

Let $J_1$ be the set of indices $j\in J$ such that $a_j\in \acl_0(a^0{}',m^0,m^1)$. As $a_j\notin S$, by Lemma~\ref{lm_strict} there is an $\LL_0$-formula $\tau_j(t, x^0,y,z)$ algebraic in $t$ and strict in $z$ such that $a_j\models \tau_j(t, a^0{}',m^0,m^1)$.

Let $J_2 = J\setminus J_1$. Then for $j\in J_2$, we have $a_j\notin \acl_0(a^0{}',m^0,m^1)$ so there is an $\LL_0$-formula $\tau_j(t, x^0,x^1,y,z)$ algebraic in $t$ and strict in $x^1$ such that $a_j\models \tau_j(t, a^0{}',a^1{}',m^0,m^1)$.

We now set $b' = bm^0m^1$ and set $\phi(a',b')$ to be the following formula
\begin{eqnarray*}
\exists v \psi(v,b)&\wedge& \bigwedge_{i\in I} \tau_i(v_i,a^0{}',m^0)\\
&\wedge& \bigwedge_{j\in J_1} \tau_j(v_j,a^0{}',m^0,m^1)\\
&\wedge& \bigwedge_{j\in J_2} \tau_j(v_j,a^0{}',a^1{}',m^0,m^1).
\end{eqnarray*}

\emph{Step $(\star)$}. By model-completeness we have that $\NN\succ \MM$. As $a'$ is $\indi{0}$ independent over $\MM$ it follows that $\MM\models \theta_\phi(b')$. Using one instance of the axiom scheme, there exists $d'\in \MM$ such that $d'\models \phi(x,b')$ with $d^0{}'\subseteq \MM_0$ and for all $j\in J_2$, all the realizations of $\tau_j(t, d',m)$ are not in $\MM_0$. Let $d$ be the tuple whose existence is stated in $\phi(d',b')$, in particular $\MM \models \psi(d,b)$. For $i\in I$, we have $d_i\in \acl_0(d^0{}'m^0)\subseteq \MM_0$. For $j\in J_2$ we already saw that $d_j\notin \MM_0$. For $j\in J_1$, as $\tau_j(t,d^0{}',m^0,m^1)$ is strict in the variable of $m^1$ and $m^1$ is $\indi 0$-independent over $\MM_0$, we have that $d_j\notin \acl_0(d^0{}',m^0)$. Recall that $m^1\indi{0} \MM_0$, so $m^1\indi{0}_{d^0{}',m^0} \MM_0$ hence $\acl_0(d^0{}',m^0,m^1)\cap \MM_0 = \acl_0(d^0{}',m^0)$, so $d_j$ cannot belong to $\MM_0$.
We conclude that $$(\MM,\MM_0)\models \psi(d,b)\wedge \bigwedge_{i\in I} d_i\in S \wedge \bigwedge_{j\in J} d_j\notin S$$
which proves that $(\MM,\MM_0)$ is existentially closed in $(\NN,\NN_0)$.
\end{proof}

\begin{rk}\label{rk_genpred}
  Notice that if we consider $\LL_0 = \set{=}$, the previous Theorem gives nothing more than the generic predicate (see\cite{CP98}). The hypothesis~\ref{H4} becomes equivalent to elimination of $\exists^{\infty}$ in that case. Note also that if $T_0$ is strongly minimal and has quantifier elimination in $\LL_0$, the conditions~\ref{H2} and~\ref{H3} are satisfied.
\end{rk}

We can forget hypothesis~\ref{H1} to get this adapted version of Theorem~\ref{model_com_gen}.

\begin{prop}\label{cor_com_gen}
  Assume that the following holds.
\begin{enumerate}
  \item[$(H_2)$] $T_0$ is model complete and for all $A$ infinite, $\acl_0(A) \models T_0$;
  \item[$(H_3)$] $T_0$ is pregeometric;
  \item[$(H_4)$] for all $\LL$-formula $\phi(x,y)$ there exists an $\LL$-formula $\theta_\phi(y)$ such that for $b\in \MM\models T$
\begin{eqnarray*}
\MM \models \theta_\phi(b) &\iff& \text{there exists $\NN\succ \MM$ and $a\in \NN$ such that}\\
 &\mbox{ }& \text{  $\phi(a,b)$ and $a$ is an $\indi 0$-independent tuple over $\MM$}
\end{eqnarray*}
\end{enumerate}

Then there exists a theory $TS$ containing $T_S$ such that
\begin{itemize}
\item every model $(\MM,\MM_0)$ of $T_S$ has a strong extension $(\MM', \MM_0')$ which is a model of $TS$, such that $\MM\prec \MM'$;
\item\label{cond2} assume that $(\MM,\MM_0)\models TS$ and $(\NN,\NN_0)$ is a model of $T_S$ which is a strong extension of $(\MM,\MM_0)$. If $\MM$ is existentially closed in $\NN$ then $(\MM,\MM_0)$ is existentially closed in $(\NN,\NN_0)$.
\end{itemize}
An axiomatization of $TS$ is given by adding to $T_S$ the following axioms, for each tuple of variables $x  = x^0x^1$,  for $\LL$-formula $\phi(x,y)$, and $\LL_0$-formulae $(\tau_i(t,x,y))_{i<k}$ which are algebraic in $t$ and strict in $x^1$,
$$\forall y (\theta_\phi(y) \rightarrow (\exists x \phi(x,y)\wedge x^0\subseteq S\wedge \bigwedge_{i<k} \forall t \ (\tau_i(t,x,y)\rightarrow t\notin S))).$$
\end{prop}
\begin{proof}
  The same proof as for Theorem~\ref{model_com_gen} works. In the proof of Theorem~\ref{model_com_gen}, the model-completeness of $T$ was used to ensure that given any model $\NN$ of $T$ extending $\MM$, then $\MM$ is existentially closed in $\NN$, which is now part of the second bullet. In the first bullet, the model $\MM'$ of $T$ extending $\MM$ is the union of an elementary chain of extensions hence is an elementary extension of $\MM$, this condition does not use the model-completeness of $T$. 
\end{proof}

\begin{rk}[A weak version of \ref{H4}]\label{rk_mcvar}
  Assume that $T,T_0$ satisfies~\ref{H1}, \ref{H2} and \ref{H3}. Assume that there is a class $\CCCC$ of $\LL$-formula such that for all $\MM\models T$, for all $\LL$-formula $\phi(x,b)$ with parameters in $\MM$, there exists a tuple $c$ from $\MM$ and formulae $\vartheta_1(x,z),\cdots,\vartheta_n(x,z)\in \CCCC$ such that $$\phi(\MM,b) = \vartheta_1(\MM,c)\cup \cdots\cup \vartheta_n(\MM,c).$$
  Assume that condition \ref{H4} holds only for formulae $\vartheta(x,z)\in \CCCC$. Then the conclusion of Theorem~\ref{model_com_gen} applies, with the axiom-scheme restricted to formulae in $\CCCC$. It is clear that the proof of the first assertion works similarly, considering only formulae in $\CCCC$. For the second assertion, the proof changes at Step $(\star)$, we need to show that there exists a realisation of $\phi(x,b')$ that satisfies the right properties using the axioms. By assumption $\phi(\MM,b') = \vartheta_1(\MM,c)\cup \cdots\cup \vartheta_n(\MM,c)$ for some $\vartheta_1(x,z),\cdots,\vartheta_n(x,z)\in \CCCC$ and tuple $c$ from $\MM$. This decomposition holds also in $\NN$ by model-completeness of $T$. Now as $a'\models \phi(x,b')$, there is some $i\leq n$ such that $a'\models \vartheta_i(x,c)$ hence $\MM\models \theta_{\vartheta_i}(c)$. Using one instance of the axiom, there exists $d'$ in $\MM$ satisfying $\vartheta_i(x,c)$, hence also $\phi(x,b')$, and that satisfies the adequat repartition of its coordinate in or out of $S$, and the end of the proof is similar. The main example for the class $\CCCC$ is the class of quasi-affine varieties in the theory $\ACF$, see Theorem~\ref{thm_genmult}.
\end{rk}

\begin{rk}[A weak converse for Theorem~\ref{model_com_gen}]
  It is possible to prove a weak converse of Theorem~\ref{model_com_gen}, using the same method as in~\cite[2.11. Proposition]{CP98}. It can be stated as follows: if \ref{H1}, \ref{H2} and \ref{H3} are satisfied and if $TS$ exists as in Theorem~\ref{model_com_gen} then:
  \begin{center}for all $\LL$-formula $\phi(x,y)$ and all $1\leq k\leq \abs{x}$, there exists an $\LL$-formula $\theta_\phi^k(y)$ such that for all tuple $b$ in an $\aleph_0$-saturated model $\MM$ of $T$,
    \begin{eqnarray*}
  \MM \models \theta_\phi^k(b) &\iff& \text{there exists some realisation $a$ of $\phi(x,b)$ in $\MM$ such that}\\
  &\mbox{ }& \text{ $a_k \notin \acl_0(\acl_T(b), a_1,\dots,a_{k-1},a_{k+1},\cdots,a_{\abs{x}})$.}
\end{eqnarray*}
\end{center}
In particular $T$ eliminates $\exists^{\infty}$.  A question one might ask is wether elimination of $\exists^{\infty}$ is a sufficient condition for the existence of the theory $TS$. The answer is no, the theory $\ACF_0$ eliminates $\exists^\infty$ but the model companion of the theory of algebraically closed fields of characteristic $0$ with a predicate for an additive subgroup is not first order axiomatisable, see Proposition~\ref{prop_car0}. On the other hand, the existence of $TS$ under the reduction of the hypothesis \ref{H4} to formulae $\phi(x,y)$ with $\abs{x} = 2$ would be a good improvement, as it would be much easier to check. A full converse is not reasonnable to ask, in view of the example of Subsection~\ref{sec_genmult}, where the model-companion exists without having the full \ref{H4} hypothesis.
\end{rk}


\begin{df}\label{def_suitabletriple}
  We say that a triple $(T, T_0,\LL_0)$ is \emph{suitable} if it satisfies the following
 \begin{enumerate}[label={$(H_{\arabic*})$}]
\item $T$ is model complete;
\item $T_0$ is model complete and for all infinite $A$, $\acl_0(A) \models T_0$;
\item[$(H_3^+)$] $\acl_0$ defines a modular pregeometry;
\item[$(H_4)$] for all $\LL$-formula $\phi(x,y)$ there exists an $\LL$-formula $\theta_\phi(y)$ such that for $b\in \MM\models T$
\begin{eqnarray*}
\MM \models \theta_\phi(b) &\iff& \text{there exists  $\NN\succ \MM$ and $a\in \NN$ such that}\\
&\mbox{ }& \text{$\phi(a,b)$ and $a$ is $\indi 0$-independent over $\MM$.}
 \end{eqnarray*}
\end{enumerate}
\end{df}

Hypothesis $(H_3^+)$ makes obsolete the notion of \emph{strong} extension. As a consequence, the theory $TS$, if it exists, is the model-companion of the theory $T_S$.

\begin{rk}\label{rk_hyp5}
Let $(T,T_0,\LL_0)$ be a suitable triple. By~\cite[Proposition 1.5]{A09}, in $T$, the relation $\indi a$ defined by $A\indi a _C B$ if and only if $\acl_T(AC)\cap \acl_T(BC) = \acl_T(C)$ satisfies \ref{EXT}, so for all $A,B,C$ subsets of $\M$ there exists $A'\equiv_C^T A$ such that $\acl_T(A'C)\cap \acl_T(BC) = \acl_T(C)$. As $\acl_0$ is modular, it follows that $\acl_T(A'C)\indi 0_{\acl_T(C)} \acl_T(BC)$.
\end{rk}
From Theorem~\ref{model_com_gen}, we immediately get the following.

\begin{prop}\label{cor_suitabletriple}
  Let $(T,T_0,\LL_0)$ be a suitable triple. Then $TS$ exists and is the model-companion of the theory $T_S$.
\end{prop}

\begin{rk} In the sense of~\cite{KTW18}, the theory $TS$ is the interpolative fusion of $T$ with the theory of generic pairs of models of $T_0$.
\end{rk}
\begin{lm}\label{amalgam_T_S}
  Let $(\MM,\MM_0)$ and $(\NN,\NN_0)$ are two models of $T_S$, such that $\MM_0\indi{0}_{\NN_0} \NN$ and $\NN_0\indi{0}_{\MM_0} \MM$.
  Then, there exists a model $(\KK,\KK_0)$ of $TS$ extending both $(\MM,\MM_0)$ and $(\NN,\NN_0)$. If furthermore $(\MM,\MM_0)$ and $(\NN,\NN_0)$ are models of $TS$, then $(\KK,\KK_0)$ is an elementary extension of both $(\MM,\MM_0)$ and $(\NN,\NN_0)$.
\end{lm}

\begin{proof}
  Let $\KK'$ be a model of $T$ extending $\MM$ and $\NN$. Now set $\KK_0' = \acl_0(\MM_0,\NN_0)$. Clearly $(\KK',\KK'_0)$ is a model of $T_S$. By hypothesis we have $\KK'_0\indi{0}_{\MM_0} \MM$ and $\KK'_0\indi{0}_{\NN_0} \NN$. Using Theorem~\ref{model_com_gen}, there is a model $(\KK,\KK_0)$ of $TS$ extending $(\KK',\KK'_0)$, $(\MM,\MM_0)$ and $(\NN,\NN_0)$. We conclude by model-completeness. 
\end{proof}

\begin{prop}\label{cor_com}
  Let $(T,T_0,\LL_0)$ be an adapted triple.
  \begin{enumerate}
    \item Let $(\MM,\MM_0)$ and $(\NN,\NN_0)$ be two models of $TS$ and $A$ be a common subset of $\MM$ and $\NN$. Then we have
      \begin{eqnarray*}
    (\MM,\MM_0)\equiv^{TS}_A (\NN,\NN_0) &\iff& \text{ there exists $f : \acl_T(A)\rightarrow \acl_T(A)$}\\
    &\mbox{}& \text{ $T$-elementary bijection over $A$,} \\
        &\mbox{}& \text{ such that $f(\MM_0\cap \acl_T(A)) = \NN_0\cap \acl_T(A)$.}
      \end{eqnarray*}

    \item For any $a,b, A$ in a monster model of $TS$
      \begin{eqnarray*}
    a\equiv^{TS}_A b  &\iff& \text{ there exists $f : \acl_T(Aa)\rightarrow \acl_T(Ab)$}\\
    &\mbox{}& \text{ a $T$-elementary bijection over $A$ with $f(a) = b$,} \\
        &\mbox{}& \text{ such that $f(S( \acl_T(Aa))) = S( \acl_T(Ab))$.}
    .
      \end{eqnarray*}
      We call such a function a \emph{$T$-elementary $\LL_S$-isomorphism between \\
     $(\acl_T(Aa), S(\acl_T(Aa))$ and $(\acl_T(Ab), S( \acl_T(Ab))$}.
      \item The completions of $TS$ are given by the $T$-elementary $\LL_S$-isomorphism types of $$(\acl_T(\emptyset), S(\acl_T(\emptyset))).$$
    \item For all $A$, $\acl_{TS}(A) = \acl_{T}(A)$.
  \end{enumerate}
\end{prop}

\begin{proof}
  \textit{(1)} The left to right implication is standard. From right to left. Note that, under hypotheses, we may assume that $A = \acl_T(A)$ is a subset of both $\MM$ and $\NN$ and that $\MM_0\cap A = \NN_0\cap A$.
  By Remark~\ref{rk_hyp5}, there exists $\MM'\equiv^T_A \MM$ such that $\MM' \indi{0}_{A} \NN$. There is an $\LL$-isomorphism $g$ between $\MM'$ and $\MM$ that fixes $A$, so we may define $\MM_0' = g^{-1}(\MM_0)$ and turn $(\MM',\MM_0')$ into a model of $TS$. By \ref{MON} and \ref{BMON} we have $\MM_0' \indi{0}_{\NN_0} \NN$. Similarly we have $\NN_0 \indi 0 _{\MM_0'} \MM'$ hence by Lemma~\ref{amalgam_T_S} there exists a model $(\KK,\KK_0)$ of $TS$ that is an elementary extension of both $(\MM',\MM_0')$ and $(\NN,\NN_0)$, hence $(\MM',\MM_0')\equiv^{TS}_A (\KK,\KK_0)\equiv^{TS}_A (\NN,\NN_0)$.

    \textit{(2)} This is similar to \textit{(1).}

    \textit{(3)} This is an obvious application of \textit{(1).}

  \textit{(4)} We only need to show that $\acl_{TS}(A)\subseteq \acl_T(A)$. Assume that $b\notin \acl_T(A)$.
  Let $(\MM,\MM_0)$ be a model of $TS$ containing $b$. There exists a model $\NN$ of $T$ and a $T$-isomorphism $f:\NN\rightarrow \MM$ over $A$ such that $\NN\indi{0}_{\acl_T(A)}\MM$. Consider $\NN_0 = f^{-1}(\MM_0)$, then $(\NN,\NN_0)$ and $(\MM,\MM_0)$ are $\LL_S$-isomorphic.
  Now set $b' = f^{-1}(b)$, we have $b'\equiv^{TS}_A b $ and $b\neq b'$ because $b\indi{0}_{\acl_T(A)} b'$ and $b\notin \acl_{T}(A)$. Since $\NN\indi{0}_{\acl_T(A)}\MM$, we may do as in \textit{(1)} and find a model of $TS$ extending both $\MM$ and $\NN$ in which the condition \textit{(3)} is satisfied. Similarly we can produce as many conjugates of $b$ over $A$ as we want inside some bigger model so $b\notin \acl_{TS}(A)$.
  \end{proof}

\begin{prop}\label{prop_type}
  Let $\M$ be a monster model of $T$. Let $\MM\prec \M$ and $\MM_0\subseteq \MM$ such that $(\MM,\MM_0)$ is a model of $TS$. Let $B\subset \MM$, and $X$ a small subset of $\M$. Let $S_{XB}\subseteq \acl_T(XB)\subset \M$ be some $\acl_0$-closed set containing $S(\acl_T(B))$ and such that:
\begin{enumerate}
  \item\label{hyp_cons} $S_{XB}\cap \MM = S(\acl_T(B))$
  \item\label{hyp_type} $\acl_T(XB)\cap \MM = \acl_T(B)$.
\end{enumerate}
Then the type (over $B$) associated to the $T$-elementary $\LL_S$-isomorphism type of $(\acl_T(XB),S_{XB})$ is consistent in $Th(\MM,\MM_0)$. 
\end{prop}
\begin{proof}
  Let $\M_0' = \acl_0(\MM_0,S_{XB})$. We have that $(\M,\M_0')$ is a model of $T_S$ and an extension of $(\MM,\MM_0)$. Indeed, $\M_0'\cap \MM = \acl_0(\MM_0, S_{XB})\cap \MM = \acl_0(\MM_0, S_{XB}\cap \MM)$ by modularity. By hypothesis~\textit{(\ref{hyp_cons})}, $S_{XB}\cap \MM = S(\acl_T(B)\subseteq \MM_0$ hence $\M_0'\cap \MM = \MM_0$. By Theorem~\ref{model_com_gen} there exists a model $(\NN,\NN_0)$ of $TS$ extending $(\M,\M_0')$ which is an elementary extension of $(\MM,\MM_0)$. Now
    \begin{align*}
   \acl_T(XB)\cap \NN_0   &= \acl_T(XB)\cap \M_0                   & \\
                          &= \acl_T(XB)\cap \acl_0(\MM_0, S_{XB})  & \\
                          &= \acl_0(S_{XB}, \acl_T(XB)\cap \MM_0)   &\text{ by modularity }\\
                          &= \acl_0(S_{XB}, \acl_T(B)\cap \MM_0)    &\text{ by \textit{(\ref{hyp_type})} }\\
                          &= \acl_0(S_{XB}, S(\acl_T(B)))          & \\
                          &= S_{XB}.
    \end{align*}
    It follows that in $(\NN,\NN_0)$, $tp^{TS}(X/B)$ is given by the $T$-elementary $\LL_S$-isomorphism type of $(\acl_T(XB), S_{XB})$.
  \end{proof}

\section{Iterating the construction}\label{chap1_part_it}

Let $T$ be an $\LL$-theory, $\LL_1,\cdots,\LL_n$ be sublanguages of $\LL$ and let $T_i = T\upharpoonright \LL_i$. Let $S_1,\cdots,S_n$ be new unary predicate and let $\LL_{S_1\ldots S_n}$ be the language $\LL\cup \set{S_1,\cdots,S_n}$. Let $T_{S_1\ldots S_n}$ be the $\LL_{S_1\ldots S_n}$-theory which models are models $\MM$ of $T$ in which $\MM_i:= S_i(\MM)$ is an $\LL_i$-substructure of $\MM$ and a model of $T_i$. The following give a condition for the existence of a model companion for $T_{S_1\ldots S_n}$.


\begin{prop}\label{prop_iterate}
  Assume inductively that $(TS_1\dots S_i, T_{i+1}, \LL_{i+1})$ is a suitable triple for $i=0,\cdots,n-1$, and let $TS_1\ldots S_{i+1}$ be the model companion of the theory ${TS_1,\ldots,S_i}_{S_{i+1}}$ of models of $TS_1,\ldots,S_i$ with a predicate $S_{i+1}$ for an $\LL_{i+1}$ submodel of $T_{i+1}$. 
  Then $TS_1\cdots S_n$ is the model-companion of the theory $T_{S_1\ldots S_n}$.
\end{prop}

\begin{proof}
  We show the following:
\begin{enumerate}
  \item every model $(\MM,\MM_1,\dots \MM_n)$ of $T_{S_1\dots S_n}$ can be extended to a model $(\NN,\NN_1,\cdots,\NN_n)$ of $TS_1\dots S_n$;
  \item every model $(\NN,\NN_1,\dots,\NN_n)$ of $TS_1\dots S_n$ is existentially closed in an extension $(\MM,\MM_1,\dots, \MM_n)$ model of $T_{S_1,\dots S_n}$.
\end{enumerate}

  (1) Start by extending $(\MM, \MM_1)$ to a model $(\NN^1,\NN_1^1)$ of $TS_1$. Then $(\NN^1,\NN_1^1,\MM_2)$ is a model of ${TS_1}_{S_2}$ so can be extended to a model $(\NN^2,\NN_1^2,\NN_2^2)$ of $TS_1 S_2$. The structure $(\NN^2,\NN_1^2,\NN_2^2)$ is also an extension of $(\MM,\MM_1,\MM_2)$. We iterate this process to end with a model $(\NN^n,\NN_1^n,\cdots,\NN_n^n)$ of $TS_1\cdots S_n$ extending $(\MM,\MM_1,\cdots,\MM_n)$.\\
(2) Let $(\NN,\NN_1,\cdots,\NN_n)$ be a model of $TS_1\cdots S_n$ and $(\MM,\MM_1,\cdots,\MM_n)$ be a model of $T_{S_1\ldots S_n}$ extending it. By (1) there exists a model $(\MM',\MM_1',\cdots,\MM_n')$ of $TS_1\cdots S_n$ extending $(\MM,\MM_1,\cdots,\MM_n)$. As $(\NN,\NN_1,\cdots,\NN_n)$ is a model of $TS_1\cdots S_n$
it is existentially closed in any model of ${TS_1\cdots S_{n-1}}_{S_n}$ extending it, in particular, it is existentially closed in $(\MM',\MM_1',\cdots,\MM_n')$ and hence also in $(\MM,\MM_1,\ldots,\MM_n)$.
\end{proof}

In a model of $TS_1\cdots S_n$, the relations between the $S_i$ are very generic. For example, it is not possible that $S_i\subseteq S_j$ for some $i,j$, since one can always extend the predicate $S_i$ by a new element which is not in $S_j$. In a sense, those generic predicates are invisible from one another. A way to impose relations between the $S_i$, is by considering, for instance, a slightly stronger version of the generic expansion by a reduct --analogously to the generic predicate in~\cite{CP98}. Consider a suitable triple $(T,T_0,\LL_0)$ and $P$ a $0$-definable predicate in $T$ such that in any model $\MM$ of $T$, $P$ is a model of $T_0$ which is a substructure of $\MM$. One may do the construction of the generic expansion by a substructure $S$ inside $P$. In that case, assume that $T_i = T_j$ for all $i,j\leq n$. One may construct $TS_1$ then add a generic substructure $S_2$ inside $S_1$ and iterate. This would be the model companion of the theory $T_{S_1\dots S_n}\cup \set{S_1\supseteq S_2 \supseteq \dots \supseteq S_n}$. One may also consider the case in which $T_i$ is not the theory of a substructure but of a structure $0$-definable in $T$.

\section{Independence relations in $T$ and $TS$}\label{sec_ind}

We set up the context for this section and Section~\ref{part_pres}. Let $(T,T_0,\LL_0)$ be a suitable triple (see Definition~\ref{def_suitabletriple} and Proposition~\ref{cor_suitabletriple}). We work in a monster model $(\M,\M_0)$ of $TS$ such that $\M$ is a monster model of $T$. In particular we fix some completion of $TS$. All small sets $A,B,C,\dots$ or models $\MM,\NN$ of $T$, or models $(\MM,\MM_0),(\NN,\NN_0)$ of $TS$ are seen as subsets of $\M$, respectively elementary substructures of $\M$ or elementary substructures of $(\M,\M_0)$. For instance we have $S(\MM) = \MM\cap S(\M) = \MM\cap \M_0 = \MM_0$. We will start with a ternary relation ($\indi T$) defined over subsets of $\M$ and construct from it a ternary relation ($\indi w$) taking into account the predicate $S(\M) = \M_0$.

We denote by $\overline{A}$ the set $\acl_T(A)$ which, as we saw, equals $\acl_{TS}(A)$.

~

\noindent \textbf{Assumption.} There exists a ternary relation $\indi{T}$ defined over subsets of $\M$, such that $\indi T\rightarrow \indi a$, where $A\indi a _C B \iff \ol{AC}\cap \ol{BC} = \ol{C}$. \\

In particular, if $A\indi T _C B$ then $\acl_T(AC)\indi 0 _{\acl_T(C)} \acl_T(BC)$, by modularity.

~

\begin{df}\label{def_weakstrong}
  We call \emph{weak independence} the relation $\indi w$ defined by
$$A\indi{w}_C B \iff A\indi{T}_C B \text{ and } S(\acl_0(\overline{AC},\overline{BC})) = \acl_0(S(\overline{AC}),S(\overline{BC})).$$
We call \emph{strong independence} the relation $\indi{st}$ defined by
$$A\indi{st}_C B \iff A\indi{T}_C B \text{ and } S(\ol{ABC}) = \acl_0(S(\overline{AC}),S(\overline{BC})).$$
Obviously $\indi{st}\rightarrow \indi w$.
\end{df}

We will show that if $\indi T$ satisfies most of the properties listed above relatively to the theory $T$, then so does $\indi w$ relatively to the theory $TS$. The property \ref{SYM} of $\indi 0$, $\indi T$ and $\indi w$ will be tacitly used throughout this chapter.

\begin{lm}\label{lm_propbase} 
If $\indi{T}$ satisfies \ref{INV}, \ref{CLO}, \ref{SYM}, \ref{EX} and \ref{MON}, then so does $\indi{w}$.
\end{lm}

\begin{proof}
  \ref{INV} is clear because  $S(\acl_0(\overline{AC},\overline{BC})) = \acl_0(S(\overline{AC}),S(\overline{BC}))$ is an $\LL_S$-invariant condition. \ref{CLO}, \ref{SYM} and \ref{EX} are trivial. 
  
  For \ref{MON}, let $A,B,C,D$ such that $A\indi{w}_C BD$. By hypothese, $A\indi{T}_C B$. Now 
  \begin{align*}
    S(\acl_0(\overline{AC},\overline{BC})) &=  S(\acl_0(\overline{AC},\overline{BCD}))\cap \acl_0(\overline{AC},\overline{BC})\\
        &=\acl_0(S(\overline{AC}),S(\overline{BCD}))\cap \acl_0(\overline{AC},\overline{BC}).
      \end{align*}
      Since $S(\ol{AC})\subseteq \acl_0(\overline{AC},\overline{BC})$, we have by modularity 
      $$\acl_0(S(\overline{AC}),S(\overline{BCD}))\cap \acl_0(\overline{AC},\overline{BC}) = \acl_0(S(\overline{AC}),S(\overline{BCD})\cap \acl_0(\overline{AC},\overline{BC})).$$ Using that $\indi T\rightarrow \indi a$, it follows from the hypotheses that $\ol{AC}\indi 0 _{\ol{C}} \ol{BCD}$ hence by \ref{BMON} of $\indi{0}$ we have $\overline{BCD}\cap \acl_0(\overline{AB},\overline{BC}) = \overline{BC}$ hence $$S(\overline{BCD})\cap \acl_0(\overline{AC},\overline{BC}) = S(\overline{BC}).$$ 
  It follows that $S(\acl_0(\overline{AC},\overline{BCD})) = \acl_0(S(\overline{AC}),S(\overline{BC}))$ and so $A\indi w _C B$.
\end{proof}

\begin{lm}\label{lm_ext}
  If $\indi T$ satisfies~\ref{EXT}, then $\indi{st}$ and $\indi w$ satisfy~\ref{EXT}.
\end{lm}

\begin{proof}
  We show that $\indi{st}$ satisfies \ref{EXT}. Let $A,B,C$ be contained in some model $(\MM,\MM_0)$ of $TS$. By \ref{EXT} for $\indi T$, there exists $A'\equiv^{T}_C A$ with $A'\indi T _C \MM$, in particular $\ol{A'C}\cap \ol{BC} = \ol C$. Using \ref{EXT} of $\indi a$ we may assume that $\overline{A'BC}\cap \MM = \overline{BC}$. Let $f : \overline{A'C}\rightarrow \overline{AC}$ be a $T$-elementary isomorphism over $C$ and $S_{A'C} := f^{-1}(S(\overline{AC}))$. Let $S_{A'BC} = \acl_0(S_{A'C}, S(\overline{BC}))$.  It is easy to see that 
\begin{itemize}
  \item $S_{A'BC}\cap \MM = S_{A'BC}\cap \overline{BC} = S(\overline{BC})$
\item $S_{A'BC}\cap \overline{A'C} = S_{A'C}$
\end{itemize} 
Using $\overline{A'BC}\cap \MM = \overline{BC}$ and the first item, the type over $BC$ defined by the pair $(\overline{A'BC},S_{A'BC})$ is consistent (see Proposition~\ref{prop_type}).  We may assume that $A'\subseteq \M$ realizes this type. From the second item, we have that $A'\equiv^{TS}_C A$, and it is clear that $S(\overline{A'BC})$ is equal to $\acl_0(S(\ol{A'C}),S(\overline{BC}))$ so $A'\indi{st}_C B$. We conclude that \ref{EXT} is satisfied by $\indi{st}$. As $\indi{st}\rightarrow \indi w$, \ref{EXT} is also satisfied by $\indi w$.
\end{proof}

\begin{lm}\label{lm_str}
  If $\indi T$ satisfies \ref{STRFINC} over algebraically closed sets, then the relation $\indi w$ satisfies~\ref{STRFINC} over algebraically closed sets.
\end{lm}
\begin{proof}
  Assume that $a \nindi{w}_C b$ and $C = \ol{C}$. If $a \nindi{T}_C b$, we have a formula witnessing \ref{STRFINC} over $C$ by hypothesis. Otherwise, assume that $a\indi{T}_C b$, set $A = \overline{Ca}$, $B = \overline{Cb}$ and assume that there exists $s\in S(\acl_0(A,B))  \setminus \acl_0(S(A),S(B))$. Let $u\in A\setminus S(A)$ and $v\in B\setminus S(B)$ be such that $s\in \acl_0(u,v)$. There exists $\LL_S$-formulae $\psi_u(y,a,c)$ and $\psi_v(z,b,c)$ isolating respectively $tp^{TS}(u/Ca)$ and $tp^{TS}(v/Cb)$ for some tuple $c$ in $C$. There is also an $\LL_0$-formula $\phi(t,y,z)$ algebraic in $t$, strict in $y$ and strict in $z$, such that $s\models \phi(t,u,v)$.\\
    
  \begin{claim} $v\notin \acl_0(S(B), C)$. \end{claim}
  \begin{proof}[Proof of the claim] Assuming otherwise, by modularity there exists singletons $s_b\in S(B)$ and $c\in C$ such that $v\in \acl_0(s_b,c)$ and so $s\in \acl_0(s_b,c,u)$. As $cu\subseteq A$, by modularity there exists a singleton $u'\in A$ such that $s\in \acl_0(s_b,u')$ and by Exchange $u'\in \acl_0(s_b, s)\cap A \subseteq S(A)$, this contradicts the hypothesis on $s$.\end{proof}
    
    In particular for any other realisation $v'$ of $\psi_v(z,b,c)$ we have $v'\notin \acl_0(S(B), C)$. Now let $\Lambda(x, b, c)$ be the following formula
    $$\exists y \exists z \exists t \psi_u(y,x,c)\wedge \psi_v(z,b,c) \wedge \phi(t,y,z)\wedge t\in S.$$

    We have that $\Lambda(x, b, c)\in tp^{TS}(a/bC)$. Assume that $a'\models \Lambda(x, b, c)$. If $a'\nindi{T}_C b$ then we are done, so we may assume that $a'\indi{T}_C b$, in particular $\overline{Ca'}\cap B = C$ as $C$ is algebraically closed. There exists $u'\in \overline{Ca'}$ and $v'\in B\setminus \acl_0(S(B),C)$ such that there is $s'\in \acl_0(u',v')\cap S$. In particular $v'\in \acl_0(s',u')$ as $\phi(t,y,z)$ is strict in $z$. Now assume that $s'\in \acl_0(S(B),S(\overline{Ca'}))$, then $v'\in \acl_0(\overline{Ca'},S(B))$ and also $v'\in B$. By modularity, 
    $$\acl_0(S(B),\overline{Ca'})\cap B = \acl_0(S(B),\ol{Ca'}\cap B) = \acl_0(S(B),C)$$
    so $v'\in \acl_0(S(B),C)$, a contradiction. We conclude that $$s'\in S(\acl_0(\ol{Ca'},B))\setminus \acl_0(S(\ol{Ca'}),S(B))$$ so $a'\nindi w _C B$.
  \end{proof}

\begin{thm}\label{thm_ind}
  Assume that $\indi{T}$ satisfies the hypotheses of Lemmas~\ref{lm_propbase}. Assume that for some subset $E$ of $\M$, the following two properties hold:
  \begin{enumerate}[label={$(A_\arabic*)$}]
    \item $\ind'$-\ref{AM} over $E$ for some $\ind'\rightarrow \indi a$, $\ind '$ satisfying \ref{MON}, \ref{SYM} and \ref{CLO};
\item\label{hyp_inter} For all $A,B,C$ algebraically closed containing $E$, if $C\indi T _E A,B$ and $A\ind' _E B$ then $$(\overline{AC},\overline{BC}) \indi{0}_{A,B} \overline{AB}.$$
\end{enumerate}
Then $\indi{w}$ satisfies $\ind'$-\ref{AM} over $E$.
\end{thm}

\begin{proof}
Let $c_1,c_2, A, B$ be in a $(\MM,\MM_0)\prec (\M,\M_0)$ such that
\begin{itemize}
\item $c_1\equiv_E^{TS} c_2$
\item $A\ind'_E B$
\item $c_1\indi{w}_E A$ and $c_2\indi{w}_E B$ 
\end{itemize}
As $\ind' $ satisfies \ref{SYM}, \ref{CLO} and \ref{MON}, we have that $A\ind' _E B  \iff \overline{AE}\ind '_{E} \overline{BE}$, hence we may assume that $A,B$ are algebraically closed and contain $E$. By hypothesis there is a $T$-elementary $\LL_S$-isomorphism $h:\overline{E c_1}\rightarrow \overline{E c_2}$ over $E$ sending $c_1$ to $c_2$. Let $C_1$ be an enumeration of $\overline{E c_1}$ and let $C_2$ be the enumeration $h(C_1)$. We have $C_1 \equiv^T_E C_2$.  

We have $C_1\indi{T}_E A$, $C_2\indi{T}_E B$ and $C_1\equiv_E^T C_2$. By $(A_1)$, there exists $C$ such that $C\equiv^T_A C_1$, $C\equiv^T_B C_2$ with $C\indi{T}_E AB$, $A\indi a _C B$, $C\indi a _B A$ and $C\indi a _A B$. We may assume that $\overline{ABC}\cap \MM = \overline{AB}$ using \ref{EXT} of $\indi a$.
There exists two $T$-elementary bijections $f : \overline{AC} \rightarrow \overline{AC_1}$ over $A$ and $g : \overline{BC}\rightarrow \overline{BC_2}$ over $B$ such that $g\upharpoonright C = h\circ (f\upharpoonright C)$.

We define $S_{AC} = f^{-1}(S(\overline{AC_1})) \subseteq \overline{AC}$ and $S_{BC} = g^{-1}(S(\overline{BC_2}))\subseteq \overline{BC}$, and set $S_{ABC} = \acl_0(S_{AB},S_{AC},S_{BC})$, with $S_{AB} = S(\overline{AB})$.
The following is easy to check, it uses that $A\indi a _C B$, $C\indi a _B A$ and $C\indi a _A B$:
\begin{itemize}
      \item $S_{AB}\cap S_{AC} = S_{AB}\cap A = S_{AC} \cap A =S(A) =:S_A$;
      \item $S_{AB}\cap S_{BC} = S_{AB}\cap B = S_{BC} \cap B = S(B) =: S_B$;
      \item $S_{AC}\cap S_{BC} = S_{AC}\cap C = S_{BC} \cap C = f^{-1}(S(C_1)) = g^{-1}(S(C_2))=: S_C$.
\end{itemize}

 Furthermore, with $S_{AB}^- = S_{AB}\cap \acl_0(A,B)$, $S_{AC}^- = S_{AC}\cap \acl_0(A, C)$ and $S_{BC}^- = S_{BC}\cap \acl_0(B,C)$, it follows from $c_1\indi{w}_E A$ and $c_2\indi{w}_E B$ that 
    \begin{enumerate}
      \item $S_{AC}^-= \acl_0(S_A,  S_C)$;
      \item $S_{BC}^-= \acl_0(S_B, S_C)$.
    \end{enumerate}
  \begin{claim}
      We have the following
     \begin{itemize}
      \item $S_{ABC}\cap \overline{AB} = S_{AB}$;
      \item $S_{ABC}\cap \overline{AC} = S_{AC}$;
      \item $S_{ABC}\cap \overline{BC} = S_{BC}$.
    \end{itemize}
  \end{claim}
  \begin{proof}[Proof of the claim]
As $A\indi a _C B$, $C\indi a _B A$ and $C\indi a _A B$, we have that $\overline{AC}\indi 0 _C \overline{BC}$, $\overline{BC}\indi 0 _B \overline{AB}$ and $\overline{AC}\indi 0 _A \overline{AB}$.
By hypothesis $(A_2)$ and \ref{TRA} of $\indi 0$ we have the following:
\begin{itemize}
  \item $(\overline{AC},\overline{BC}) \indi{0}_{A,B} \overline{AB}$;
  \item $(\overline{AB},\overline{BC}) \indi{0}_{A,C} \overline{AC}$;
  \item $(\overline{AC},\overline{AB}) \indi{0}_{B,C} \overline{BC}$.
\end{itemize}

    In order to prove the first item of the claim, by modularity, it suffices to show that $\acl_0(S_{AC},S_{BC})\cap \overline{AB} \subseteq S_{AB}$. We will in fact show that 
    $$\acl_0(S_{AC},S_{BC})\cap \overline{AB} = S_{AB}^-.$$
    We have that $(\overline{AB},\overline{BC}) \indi{0}_{A,C} \overline{AC}$. Since $S_{AC}^- = S_{AC}\cap \acl_0(A,C)$ and $S_{BC}\subseteq \overline{BC}$ we deduce $S_{AC}\indi{0}_{S_{AC}^-} \overline{AB}, S_{BC}$. Now since $S_{AC}^-= \acl_0(S_A, S_C)$ we can use \ref{BMON} of $\indi{0}$ and the fact that $S_C\subseteq S_{BC}$ to get  $$S_{AC}\indi{0}_{S_A,S_B, S_{BC}} \overline{AB}.$$
    
    On the other hand, $\overline{BC}\cap \overline{AB} = B$ so $S_{BC} \indi{0}_{S_B} \overline{AB}$. Using \ref{BMON} of $\indi 0$ we also have that $S_{BC} \indi{0}_{S_A,S_B} \overline{AB}$ so using \ref{TRA} of $\indi{0}$ it follows that $(S_{AC},S_{BC}) \indi{0}_{S_A,S_B} \overline{AB}$.

    For the second item, it is sufficient to prove that $\acl_0(S_{AB},S_{BC})\cap \overline{AC}\subseteq S_{AC}$. We do similarly as before paying attention to the fact that $S_{AB}$ and $S_{AC}$ do not play a symmetric role. We get first that $S_{BC}\indi{0}_{S_{BC}^-}( \overline{AC}, S_{AB})$ using $(\overline{AC},\overline{AB}) \indi{0}_{B,C} \overline{BC}$. Now $S_{BC}^- = \acl_0(S_B,S_C)$, so we deduce $S_{BC}\indi{0}_{S_C,S_B}( \overline{AC}, S_{AB})$ and by \ref{BMON} of $\indi 0$ and the fact that $S_B,S_A\subseteq S_{AB}$ we deduce 
     $$S_{BC}\indi{0}_{S_C, S_A,S_{AB}} \overline{AC}.$$ 
     Now by \ref{BMON} of $\indi 0$, we have $S_{AB}\indi{0}_{S_A,S_C} \overline{AC}$. We conclude using \ref{TRA} of $\indi 0$ that $(S_{AB},S_{BC})\indi{0}_{S_A,S_C} \overline{AC}$. The proof of the last assertion is similar. 
   \end{proof}

     We know that $\overline{ABC}\cap \MM = \overline{AB}$. Moreover, it follows from the first point of the claim that $S_{ABC}\cap \MM = S_{ABC}\cap \overline{AB} = S_{AB}$. Consequently, by Proposition~\ref{prop_type}, the type in the sense of the theory $TS$ defined by the pair $(\overline{ABC}, S_{ABC})$ is consistent, so we may consider that it is realised in $(\M,\M_0)$, by say $C$. It follows that $C = \overline{Ec}$ with $c$ such that $c\equiv_A^{TS} c_1$ and $c \equiv_B^{TS} c_2$.
     What remains to show is that $C\indi{w}_E A,B$. We already have that $C\indi{T}_E A,B$ so we will prove that $$S(\acl_0(C, \overline{AB})) = \acl_0(S(C),S(\overline{AB})).$$
     By modularity, it suffices to show that $\acl_0(S_{AC},S_{BC})\cap \acl_0(C,\overline{AB})\subseteq \acl_0(S_C, S_{AB})$. We in fact prove that $(S_{AC},S_{BC})\indi{0}_{S_A,S_B,S_C}( \overline{AB},C)$.
     As before, using $(\overline{AB},\overline{BC}) \indi{0}_{A,C} \overline{AC}$ we have that $S_{AC}\indi{0}_{S_{AC}^-} (\overline{AB},\overline{BC})$, so as $S_{AC}^-= \acl_0(S_A , S_C)$ we have $$S_{AC}\indi{0}_{S_A, S_C}(\overline{AB},S_{BC},C).$$ Using \ref{BMON} of $\indi 0$, we have $$S_{AC}\indi{0}_{S_A,S_B, S_C,S_{BC}}( \overline{AB},C).$$
     On the other hand, from $(\overline{AC},\overline{AB}) \indi{0}_{B,C} \overline{BC}$ and \ref{MON} of $\indi 0$, we have that $\overline{BC}\indi 0 _{B,C}( \overline{AB},C)$. It follows that $S_{BC}\cap \acl_0(\overline{AB},C)\subseteq S_{BC}^- = \acl_0(S_B,S_C)$ so $S_{BC}\indi{0}_{S_B,S_C} (\overline{AB},C)$. Using \ref{BMON} of $\indi 0$ we have $$S_{BC}\indi{0}_{S_B,S_A, S_C}( \overline{AB},C).$$
     Now using \ref{TRA} of $\indi 0$, we get $(S_{AC}, S_{BC})\indi{0}_{S_A,S_B, S_C} (\overline{AB},C)$.
    \end{proof}

\begin{lm}\label{lm_wit_str}
Assume that $a\nindi w _C b$ and $a\indi T_C b$ with $C=\ol C$. Then there is a formula $\Lambda(x,b,c)\in tp(a/Cb)$ such that for all sequence $(b_i)_{i<\omega}$ such that
\begin{enumerate}
  \item $b_i\equiv^{TS}_C b$ for all $i<\omega$,
\item $b_i\indi a _C b_j$ and $S(\acl_0(\overline{Cb_i},\overline{Cb_j})) = \acl_0(S(\overline{Cb_i}),S(\overline{Cb_j}))$ for all $i,j<\omega$,
\end{enumerate}
 the partial type $\set{ \Lambda(x,b_i,c)\mid i<\omega}$ is inconsistent.
\end{lm}

\begin{proof}
Let $A = \overline{Ca}$, $B = \overline{Cb}$. As $a\nindi w _C b$ there exists $s\in S(\acl_0(A,B))  \setminus \acl_0(S(A),S(B))$.
As we saw in the proof of Lemma~\ref{lm_str}, there exist $u\in A\setminus S(A)$, $v\in B\setminus S(B)$ and $\LL_S(C)$-formulae $\psi_u(y,a)$ algebraic in $y$ and $\psi_v(z,b)$ algebraic in $z$, satisfied respectively by $u$ and $v$. There is also an $\LL_0$-formula $\phi(t,y,z)$ algebraic in $t$, strict in $y$ and strict in $z$, such that $s\models \phi(t,u,v)$. Again, as $v\notin \acl_0(S(B), C)$ and $\psi_v(z,b)$ isolates the type $tp^{TS}(v/Cb)$, every $v'$ satisfying $\psi_v(z,b)$ will satisfy $v'\notin \acl_0(S(B), C)$. Let $\Lambda(x, b, c)\in tp^{TS}(a/Cb)$ be the following formula, for a tuple $c$ from $C$
$$\exists y \exists z \exists t \psi_u(y,x)\wedge \psi_v(z,b) \wedge \phi(t,y,z)\wedge t\in S.$$
As we saw in the proof of Lemma~\ref{lm_str}, it witnesses \ref{STRFINC} over $C$. Note that if $b'\equiv^{TS}_C b$, then no realization of $\psi_v(y,b')$ is in $\acl_0(S(\ol{Cb'}),C)$.

Now let $(b_i)_{i<\omega}$ be as in the hypothesis. By contradiction, assume that $\set{ \Lambda(x,b_i,c)\mid i<\omega}$ is consistent, and realised by some $a'$. Assume that $\psi_u(t,a')$ does not have more than $k$ distinct realisations. As $$\bigwedge_{i<k+1} \Lambda(a', b_i,c)$$
is consistent, there is $u'\in \overline{Ca'}$ and $i<j<k+1$ such that $v_i,v_j$ are two realisations of $\psi_v(z,b_i)$ and $\psi_v(z,b_j)$ respectively  --we assume $i = 1,j=2$ for convenience-- and such that there exist $s_1\in \acl_0(u',v_1)\cap S$ and $s_2\in \acl_0(u', v_2)\cap S$. As $v_2\notin\acl_0(S(\ol{Cb_2}), C)$ it follows that $v_2\notin\acl_0(u')$, hence $u'\in \acl_0(s_2, v_2)$ so $s_1\in \acl_0(s_2,v_1,v_2)$. By modularity, it means that there is some $w\in \acl_0(v_1,v_2)$ such that $s_1\in \acl_0(s_2,w)$. We have that $w\in \acl_0(s_1,s_2)$, so $w\in \acl_0(v_1,v_2)\cap S$.
As $S(\acl_0(\overline{Cb_1}), \acl_0(\overline{Cb_2})) =  \acl_0(S(\acl_0(\overline{Cb_1}),S( \acl_0(\overline{Cb_2})))$ there is some $s_1^b\in S(\overline{Cb_1})$ and $s_2^b\in S(\overline{Cb_2})$ such that $w\in \acl_0(s_1^b, s_2^b)$. Now, as $v_1\notin C$, it follows that $v_1\notin \acl_0(v_2)$ hence $v_1\in \acl_0(w, v_2)$, and so $v_1\in \acl_0(s_1^b,s_2^b, v_2)$. So there is $v_2'\in \acl_0(s_2^b,v_2)\subseteq \overline{Cb_2}$ such that $v_1\in \acl_0(s_1^b,v_2')$. It follows that $v_2'\in \acl_0(s_1^b, v_1)$ so $v_2'\in \overline{Cb_1}\cap \overline{Cb_2} = C$, hence $v_2'\in C$. Now $v_1\in \acl_0(S(\overline{Cb_1}),C)$ and this is a contradiction. 
\end{proof}

\begin{lm}\label{lm_finsat}
  Let $\ind$ be a relation satisfying \ref{SYM}, \ref{MON}, \ref{EX} and \ref{STRFINC} over $C$. If $tp^{T}(a/Cb)$ is finitely satisfiable in $C$ then $a\ind _C b$.
\end{lm}
\begin{proof}
  Indeed, assume $a\nind _C b$ then by \ref{STRFINC} there is a formula $\phi(x,b)\in tp^{T}(a/Cb)$ such that if $a'\models \phi(x,b)$ then $a'\nind _C b$. As $tp^{T}(a/Cb)$ is finitely satisfiable in $C$ there is $c\in C$ such that $c\models \phi(x,b)$, so $c\nind _C b$, so by \ref{SYM} and \ref{MON} $b\nind _C C$ which contradicts \ref{EX}.
\end{proof}

\begin{lm}\label{lm_wit}
  Assume that $\indi T$ satisfies the hypothesis of Lemma~\ref{lm_propbase} and~\ref{lm_str}. If $\indi T$ satisfies \ref{WIT}, then so does $\indi w$.
\end{lm}
\begin{proof}
  Assume that $a\nindi w_\MM b$, and let $\Lambda(x,b,m)$ be as in Lemma~\ref{lm_wit_str} and set $p(x) = tp^{TS}(a/\MM b)$, $p_\LL = p\upharpoonright\LL=tp^{T}(a/\MM b)$. Let $q(x)$ be a global extension of $tp^{TS}(b/\MM)$ finitely satisfiable in $\MM$, $q_\LL = q\upharpoonright \LL$. It is clear that $q_\LL$ is finitely satisfiable in $\MM$. Let $(b_i)_{i<\omega}$ be a sequence in $\M$ such that $b_i\models q\upharpoonright \MM b_{<i}$ for all $i<\omega$. Observe that for $j<i$ we have $tp^{TS}(b_i/\MM b_j)$ is finitely satisfiable in $\MM $. By hypothesis, $\indi w$ satisfies  in particular \ref{SYM}, \ref{MON}, \ref{EX}, and \ref{STRFINC} over models, hence by Lemma~\ref{lm_finsat}, $b_i\indi w_\MM b_j$. In particular $b_i\indi a _\MM b_j$ and $S(\acl_0(\overline{\MM b_i},\overline{\MM b_j})) = \acl_0(S(\overline{\MM b_i}),S(\overline{\MM b_j}))$ for all $i,j<\omega$.
  If $\set{\Lambda(x,b_i,m)\mid i<\omega}$ is inconsistent, we conclude. If $\set{\Lambda(x,b_i,m)\mid i<\omega}$ is consistent, by Lemma~\ref{lm_wit_str} we have $a\nindi T _\MM b$. Now also $b_i\models q_\LL\upharpoonright \MM b_{<i}$, hence as $\indi T$ satisfies \ref{WIT}, we conclude. 
 \end{proof} 
 
\begin{lm}\label{lm_bmon}
Assume that $\indi T$ satisfies \ref{BMON}. The following are equivalent.
\begin{enumerate} 
\item\label{hyp_nots} $\indi w$ satisfies \ref{BMON};
\item\label{hyp_bmon} For all algebraically closed sets $A,B,C,D$ such that $A,B,D$ contain $C$ and $A\indi T _C BD$, the following holds
$$\acl_0(A,\overline{BD})\cup \overline{AD} =\acl_0(\overline{AD},\overline{BD}).$$
\end{enumerate}
In particular if $\acl_0$ is trivial or if $\acl_0 = \acl_T$ then $\indi w$ satisfies \ref{BMON}.
\end{lm}

\begin{proof}
  Assume that there exist $A,B,C,D$ that do not satisfy \textit{(\ref{hyp_bmon})}. Let $w\in  \acl_0(\overline{AD},\overline{BD})\setminus (\acl_0(A,\overline{BD})\cup \overline{AD})$, and  $S_0:= S(\acl_T(\emptyset))$. We define $S_{ABD} = \acl_0(S_0,w)$. The type (over $\emptyset$) defined by the pair $(\overline{ABD},S_{ABD})$ is consistent. As $S_{ABD}\cap \acl_0(A,\overline{BD}) =S_{ABD}\cap A =S_{ABD}\cap \overline{BD} = S_0$ and $A\indi{T}_C BD$ we have that $A\indi w _C BD$. Now $w\in S_{ABD}\cap  \acl_0(\overline{AD},\overline{BD})$ whereas $S_{ABD}\cap  \overline{AD} = S_{ABD}\cap \overline{BD}=S_0$, hence
$$S_0 = \acl_0(S_{ABD}\cap  \overline{AD}, S_{ABD}\cap \overline{BD})\subsetneq S_{ABD}\cap  \acl_0(\overline{AD},\overline{BD}).$$
It follows that $A\nindi w _D B$, so $\indi w$ doesn't satisfies \ref{BMON}. 

Conversely if $\indi w$ doesn't satisfies \ref{BMON}, it means that there exist $A,B,C,D$ such that $A\indi w_C BD$ and $A\nindi w _{CD} B$. We may assume that $A,B,D$ are algebraically closed and contains $C$. As $\indi T$ satisfies \ref{BMON} we have that $$S(\acl_0(\overline{AD},\overline{BD})) \supsetneq \acl_0(S(\overline{AD}), S(\overline{BD})).$$
Let $w$ be in $S(\acl_0(\overline{AD},\overline{BD})) \setminus \acl_0(S(\overline{AD}), S(\overline{BD}))$. As $w\in S$ we have that $w\notin \overline{AD}$ and $w\notin \overline{BD}$. It remains to show that $w\notin \acl_0(A,\overline{BD})$. Assume that $w\in \acl_0(A,\overline{BD})$. As $w\in S$ we have that $w\in S(\acl_0(A,\overline{BD}))$. From $A\indi w_C BD$ we have that $S(\acl_0(A,\overline{BD})) = \acl_0(S(A),S(\overline{BD}))$ so $w\in \acl_0(S(A),S(\overline{BD}))$ which contradicts that $w\notin \acl_0(S(\overline{AD}), S(\overline{BD}))$. So it follows that $w\in \acl_0(\overline{AD},\overline{BD})\setminus (\acl_0(A,\overline{BD})\cup \overline{AD} )$.
\end{proof}

\begin{lm}\label{lm_propst}
  Assume that $\indi T$ satisfies \ref{INV}, \ref{FIN}, \ref{SYM}, \ref{CLO}, \ref{MON}, \ref{BMON},\ref{TRA}, \ref{EXT} then so does $\indi{st}$. Furthermore, for any $E = \ol{E}$, if $\indi T$ satisfies \ref{STAT} over $E = \ol{E}$, so does $\indi{st}$. 
\end{lm}
\begin{proof}
  \ref{INV}, \ref{FIN}, \ref{SYM}, \ref{CLO} are trivial.
 \ref{EXT} is Lemma~\ref{lm_ext}. It remains to show \ref{MON}, \ref{BMON}, \ref{TRA} and \ref{STAT} over algebraically closed sets.\\

 \ref{MON}. Assume that $A\indi{st}_C BD$. We only need to check that $S(\overline{ABC}) = \acl_0(S(\overline{AC}),S(\overline{BC})$. We have 
   \begin{align*}
   S(\overline{ABC}) &= \acl_0(S(\overline{AC}),S(\overline{BCD}))\cap \overline{ABC} & \\
   &= \acl_0(S(\overline{AC}), S(\overline{BCD})\cap \overline{ABC}  &\text{(by modularity)} \\
   &=  \acl_0(S(\overline{AC}),S(\overline{BC})) &\text{ as $\overline{BCD}\cap \overline{ABC}=\overline{BC}$ ($\indi T\rightarrow \indi a$)}. & 
   \end{align*}

   \ref{BMON}. If $A\indi{st}_C BD$ then by \ref{BMON} of $\indi{T}$ we have $A\indi{T}_{CD} B$. As $S(\overline{ABCD}) = \acl_0(S(\overline{CA}),S(\overline{CBD})$, in particular $$S(\overline{ABCD}) \subseteq \acl_0(S(\overline{ACD}), S(\overline{BCD})) \subseteq S(\overline{ABCD})$$ so $A\indi{st}_{CD} B$.\\

       \ref{TRA}. Assume that $A \indi{st} _{CB} D$ and $B\indi{st}_C D$. By \ref{CLO}, we may assume that $A = \overline{ABC},B=\overline{BC}, D=\overline{CD}$. By \ref{MON}, it is sufficient to show that $A\indi{st}_C D$. We have $A\indi{T}_C D$ by \ref{TRA} of $\indi{T}$. We show that $S(\overline{AD}) = \acl_0(S(A), S(D)$. By $A\indi{st} _B D$ we have $S(\overline{AD}) = \acl_0(S(A),S(\overline{BD}))$. By $B\indi{st} _C D$, $S(\overline{BD}) = \acl_0(S(B),S(D))$ hence $S(\overline{AD}) = \acl_0(S(A),S(B),S(D)) = \acl_0(S(A),S(D)$.\\

 \ref{STAT}. 
Assume that $c_1\indi{st} _E A$ and $c_2 \indi{st}_E A$ and $c_1\equiv_E ^{TS} c_2$. We may assume that $A$ is algebraically closed and contains $E$. There is a $T$-elementary $S$-preserving map $f : \ol{Ec_1} \rightarrow \ol{Ec_2}$ over $E$. By \ref{STAT} over $E$, we can extend $f$ to $\tilde f : \ol{Ac_1} \rightarrow \ol{Ac_2}$ $T$-elementary over $A$. But as $S(\ol{Ac_1})) = \acl_0(S(\ol{Ec_1}), S(A))$ and $S(\ol{Ac_2}) = \acl_0(S(\ol{Ec_2}), S(A))$, $\tilde f$ preserves $S$, so $c_1\equiv_B^{TS} c_2$.
\end{proof}

\section{Preservation of $\NSOP{1}$, simplicity and stability}\label{part_pres}

In this section, we use the results of the previous section to prove that if $T$ is $\NSOP{1}$ and $T$ satisfies an additional hypothesis then $TS$ is also $\NSOP{1}$. This additional hypothesis (namely \hyperlink{hypn}{$(A)$} below) translates how $\indi 0$ in the reduct $T_0$ is controlled by $\indi T$ in $T$.  We work in the same context as the previous section, with small sets and small models in a monster model for $TS$, when $(T,\LL_0, T_0)$ is a suitable triple.

\begin{thm}~\label{thm_KF}
Assume that $(T,\LL_0,T_0)$ is a suitable triple. Assume that $T$ is $\NSOP{1}$ and that $\indi T$ is the Kim-independence relation in $T$. If 
\begin{enumerate}
  \item[$(A)$]  \hypertarget{hypn} all $\MM \models T$ and $A,B,C$ algebraically closed containing $\MM$, if $C\indi T _\MM A,B$ and $A\indi T_\MM B$ then $$(\overline{AC},\overline{BC}) \indi{0}_{A,B} \overline{AB}.$$
\end{enumerate}
Then $TS$ is $\NSOP{1}$ and the Kim-independence relation in $TS$ is given by $\indi w$, i.e. the relation
$$A\indi T _\MM B \text{ and } S(\acl_0(\overline{A\MM},\overline{B\MM})) = \acl_0(S(\overline{A\MM}),S(\overline{B\MM})).$$
\end{thm}

\begin{proof}
  From~\cite{KR17}, if $T$ is $\NSOP{1}$ the Kim-independence $\indi T$ satisfies  \ref{INV}, \ref{SYM}, \ref{MON}, \ref{EX} and \ref{STRFINC} all over models. Furthermore, by~\cite[Theorem 2.21]{KrR18}, it also satisfies $\indi T$-\ref{AM} over models. By Lemmas~\ref{lm_propbase},~\ref{lm_str} and Theorem~\ref{thm_ind}, all these properties are also satisfied over models by $\indi w$ (relatively to the theory $TS$). By Proposition 5.3 in~\cite{CR16}, $TS$ is $\NSOP{1}$. As $\indi T$ satisfies \ref{WIT}, so does $\indi w$ by Lemma~\ref{lm_wit}. Using~\cite[Theorem 9.1]{KR17} (and~\cite[Remark 9.2]{KR17}), it follows that $\indi w$ and Kim-independence in $TS$ coincide over models.
\end{proof}

The results of the previous section give more than the previous Theorem. Indeed, most of the nice features that may happen in $T$ for $\indi T$ are preserved when expanding $T$ to $TS$. For instance, if $\indi T$ is defined over every small base set, so is $\indi w$. If the independence theorem in $T$ is satisfied by $\indi T$ not only over models but over a wider class of small sets then the same holds in $TS$ for $\indi w$. We summarize these features in the next result.

\begin{thm}\label{thm_conserve}
Assume that $(T,\LL_0, T_0)$ is a suitable triple. Assume that there is a ternary relation $\indi{T}$ over small sets of a monster model of $T$ that satisfies
\begin{itemize}
\item \ref{INV};
\item \ref{SYM};
\item \ref{CLO};
\item \ref{MON};
\item \ref{EX};
\item \ref{EXT};
\item \ref{STRFINC} over $E$ for $E=\ol E$;
\item $\ind '$-\ref{AM} over $E$ for $E=\ol E$, where $\ind'$ is such that $\indi T \rightarrow \ind' \rightarrow \indi a$ and $\ind'$ satisfies~\ref{MON},~\ref{SYM} and \ref{CLO}; 
\item[$(A)$] For $E=\ol{E}$ and $A,B,C$ algebraically closed containing $E$, if $C\indi T _E A,B$ and $A\indi T _E B$ then $\overline{AC}\indi 0 _C \overline{BC}$ and $$(\overline{AC},\overline{BC}) \indi{0}_{A,B} \overline{AB};$$
\item \ref{WIT}.
\end{itemize}
(In particular $T$ is $\NSOP{1}$, and $\indi{T}$ coincide with Kim-independence over models of $T$, by~\cite[Proposition 5.3]{CR16} and~\cite[Theorem 9.1]{KR17}). \\

Then any completion of $TS$ is $\NSOP{1}$ and $\indi{w}$ and the Kim-forking independence relation in $TS$ coincide over models. Furthermore $\indi w$ satisfies all these properties, relatively to the theory $TS$.
\end{thm}

Finally, using \cite[Proposition 8.8]{KR17} we give a condition on $(T,T_0,\LL_0)$ that characterizes the simplicity of $TS$, assuming that $T$ satisfies the hypotheses of Theorem~\ref{thm_conserve}.

\begin{cor}\label{cor_notsimple}
Let $(T,\LL_0,T_0)$ be a suitable triple satisfying all the assumptions of Theorem~\ref{thm_conserve}. The following are equivalent.
\begin{enumerate} 
\item\label{hyp_nots} Any completion of $TS$ is not simple 
\item\label{hyp_bmon} $T$ is not simple or there exist algebraically closed sets $A,B,C,D$ such that $A,B,D$ contain $C$ and $A\indi T _C BD$, and such that
$$\acl_0(A,\overline{BD})\cup \overline{AD} \neq \acl_0(\overline{AD},\overline{BD}).$$
\end{enumerate}
In particular if $\acl_0$ is trivial or if $\acl_0 = \acl_T$ the theory $TS$ is simple if and only if $T$ is simple. If $TS$ is simple, $\indi w$ is forking independence over models.
\end{cor}

\begin{proof}
  From Theorem~\ref{thm_conserve}, we know that the relation $\indi w$ is Kim-independence over models. By \cite[Proposition 8.8]{KR17}, $TS$ is simple if and only if $\indi w$ satisfies \ref{BMON}. The equivalence follows from Lemma~\ref{lm_bmon}. The fact that Kim-independence and forking independence coincide is \cite[Proposition 8.4]{KR17}.
\end{proof}

\begin{cor}\label{cor_iterateNSOP}
  Assume that $T$ is a complete $\LL$-theory and $\LL_1,\dots,\LL_n$ are sublanguages of $\LL$. Let $T_1 = T\upharpoonright\LL_1, \dots,T_n = T\upharpoonright \LL_n$ such that $(TS_1\dots S_i, T_{i+1}, \LL_{i+1})$ is a suitable triple for each $i=0,\cdots,n-1$. By Proposition~\ref{prop_iterate}, let $TS_1\dots S_n$ be the model companion of the theory of models of $T$ with a predicate $S_i$ for an $\LL_i$ substructure. 
\begin{enumerate}
\item Assume that $T$ is $\NSOP{1}$, with Kim-independence $\indi T$ in $T$ and that for all $i$ we have (for $A,B,C$ algebraically closed containing $\MM \models T$)
\begin{center} if $C\indi T _\MM A,B$ and $A\indi T_\MM B$ then $(\overline{AC},\overline{BC}) \indi{i}_{A,B} \overline{AB}.$
\end{center}
Then $TS_1\dots S_n$ is $\NSOP{1}$ and Kim-independence in $TS$ is given by 
$$ A\indi T_\MM B \mbox{ and for all $i\leq n$ } S_i(\acl_i(\overline{A\MM},\overline{B\MM})) = \acl_i(S_i(\overline{A\MM}),S_i(\overline{B\MM}))$$
(for $\acl_i$, $\indi i$ the algebraic closure and independence in the sense of the pregeometric theory $T_i$).

\item If there exists $\indi T$ that satisfies the hypotheses of Theorem~\ref{thm_conserve} (relatively to each theory $T_i$), then $TS_1\dots S_n$ is $\NSOP{1}$ and the relation
$$ A\indi T_C B \mbox{ and for all $i\leq n$ } S_i(\acl_i(\overline{AC},\overline{BC})) = \acl_i(S_i(\overline{AC}),S_i(\overline{BC})) $$
 agrees with Kim-independence over models. Furthermore this relation satisfies all the properties listed in Theorem~\ref{thm_conserve}.
 \end{enumerate}
\end{cor}

\begin{prop}~\label{prop_stable}
  If $T$ is stable and $\acl_0 = \acl_T$, then the theory $TS$ is stable.
\end{prop}
\begin{proof}
  By Corollary~\ref{cor_notsimple}, $TS$ is simple and $\indi w$ is the forking independence. As $\acl_T = \acl_0$ it follows that $\indi{st} = \indi w$, hence as $\indi T$ is stationnary over models, so is $\indi w$ by Lemma~\ref{lm_propst}. The stability of $TS$ follows since forking independence is stationnary over models.
\end{proof}

\begin{rk}
Assume that $T$ is stable and that $\acl_0$ is trivial, then $TS$ is not necessary stable. From Corollary~\ref{cor_notsimple}, $TS$ is simple and $\indi w$ is forking independence. As $\acl_0$ is trivial, we have $\indi w = \indi T$, (with $\indi T$ forking independence in $T$) which is not likely to be stationnary. The easiest example of a reduct $T_0$ for which $\acl_0$ is trivial is the particular case of $\LL_0 = \set{=}$. Then $TS$ is the theory of the generic predicate on $T$ (see Remark~\ref{rk_genpred} and \cite{CP98}), which does not preserve stability. Indeed~\cite[(2.10) Proposition, Errata]{CP98} gives a sufficient condition on $T$ so that $TS$ have the independence property (hence is unstable): there exists a model $\MM$ of $T$ and two elements $a$ and $b$ such that $b\indi u _\MM a$ and $\ol{\MM ab}\neq \ol{\MM a}\cup\ol{\MM b}$. It follows that adding a generic predicate to an algebraically closed field result in a simple unstable theory (take $a$ and $b$ two generics independent over $\MM$).
\end{rk}

\begin{rk}[Mock stability]
  A theory $T$ is \emph{mock stable} if there is a relation satisfying \ref{INV}, \ref{FIN}, \ref{CLO}, \ref{SYM}, \ref{MON}, \ref{BMON}, \ref{TRA}, \ref{EXT}, \ref{STAT} over models. In the original definition of mock stability~\cite{Ad08}, Adler asks for slightly different properties but it is an easy exercice to check that our set of properties is equivalent to the one in~\cite{Ad08}. Using Lemma~\ref{lm_propst}, if $T$ is mock stable then so is $TS$.
\end{rk}

\section{Examples of generic expansions by a reduct}\label{part_ex}

\subsection{Generic vector subspace over a finite field}
Let $\F_q$ be a finite field. In this subsection, we let $\LL_0 = \set{(\lambda_\alpha)_{\alpha\in \F_q}, +, 0}$, and $\LL$ a language containing $\LL_0$. We let $T$ be a complete $\LL$-theory which contains the $\LL_0$-theory $T_0$ of infinite-dimensional $\F_q$-vector spaces. For $A$ a subset of a model of $T$, the set $\acl_0(A)$ is the vector space spanned by $A$, and we denote it by $\vect{A}$. 
Let $\LL_V = \LL\cup \set{V}$, with $V$ a unary predicate and $T_V$ the $\LL_V$-theory whose models are the models of $T$ in which $V$ is an infinite vector subspace.

\textbf{Definability and notations}. For $\alpha = \alpha_1,\dots , \alpha_n\in \F_q$ and any $n$-tuple $x$ of variables let $\lambda_\alpha(x)$ be the term 
$$\lambda_{\alpha_1}(x_1)+\dots +\lambda_{\alpha_n}(x_n).$$
Let $z$ be a tuple of variables of length $s = q^{n} - 1$ and $z' = z_0z$ a tuple of length $s+1 = q^{n}$. Let $\psi(t)$ be any $\LL_V$-formula, $t$ a single variable. We fix an enumeration $\alpha^1,\dots,\alpha^s$ of $(\F_q)^n\setminus (0,\dots,0)$. We denote by 
\begin{align*}
  &z = \vect{x}_0 & &\text{ the formula }  &   &\bigwedge_{i=1,\dots,s} z_i = \lambda_{\alpha^i}(x)\\
  &z'=\vect{x}  &   &\text{ the formula }  &    &z_0 = 0\wedge z = \vect{x}_0\\
  &t\in \vect{x} &    &\text{ the formula }    &   &\forall z' \left(z' = \vect{x}\rightarrow \bigvee_{i=0}^{s} t = z_i\right)\\
  &t\in \vect{xy}\setminus \vect{y} &    &\text{ the formula }    &   &t\in \vect{xy}\wedge \neg t\in \vect{y}\\
  &\vect{x}\cap \psi = \vect{y}   &   &\text{ the formula } 
  &    &\forall t \left( t\in \vect{x}\wedge  \psi(t) \leftrightarrow t\in \vect{y}\right).
\end{align*}
The formulae above have the obvious meaning, for instance, for any $a,b$ in a model of $T$, if $\MM\models b = \vect{a}_0$ then $b$ is an enumeration of all non-trivial $\F_q$-linear combinations of $a$.

The following is~\cite[Lemma 2.3]{CP98}:
\begin{fact}\label{fact_ChaPil}
  Assume that $T$ is a theory that eliminates the quantifier $\exists^{\infty}$. 
  Then for any formula $\phi(x,y)$ there is a formula $\theta_\phi(y)$ such that in any $\aleph_0$-saturated model $\MM$ of $T$ the set $\theta_\phi(\MM)$ consists of tuples $b$ from $\MM$ such that there exists a realisation $a$ of $\phi(x,b)$ with $a_i\notin \acl_T(b)$ for all $i$.
\end{fact}

\begin{thm}\label{thm_gensubvect}
If $T$ is model complete and eliminates the quantifier $\exists^\infty$, then $(T,T_0,\LL_0)$ is a suitable triple. It follows that the theory $T_V$ admits a model companion, which we denote by $TV$. 
\end{thm}

\begin{proof}
  We have to show that the triple $(T, T_0, \LL_0)$ is suitable, the existence of the model-companion then follows from Proposition~\ref{cor_suitabletriple}.
We check the conditions of Definition~\ref{def_suitabletriple}: 
\begin{enumerate}[label={$(H_\arabic*)$}]
\item $T$ is model complete;
\item $T_0$ model complete and for all infinite $A$, $\vect{A} \models T_0$;
\item[($H_3^+$)] $\vect{\cdot}$ defines a modular pregeometry;
\item[($H_4$)] for all $\LL$-formula $\phi(x,y)$ there exists an $\LL$-formula $\theta_\phi(y)$ such that for $b\in \MM\models T$
\begin{eqnarray*}
\MM \models \theta_\phi(b) &\iff& \text{there exists a saturated $\NN\succ \MM$ and $a\in \NN$ such that}\\ 
 &\mbox{ }& \text{ $\phi(a,b)$ and $a$ is $\indi 0$-independent over $\MM$.}  
\end{eqnarray*}
\end{enumerate}
Condition $(H_1)$ holds by hypothesis. Conditions $(H_2)$ and $(H_3^+)$ are also clear, these are basic properties of the theory of infinite dimensional vector spaces. As $A$ is infinite, $\vect{A}$ is an infinite dimensional $\F_q$-vector space.

We prove condition $(H_4)$. Let $\phi(x,y)$ be an $\LL$-formula.
For some tuple of variables $z$ of suitable length, let $\tilde{\phi}(z,y)$ be the following formula
$$\exists x\  z = \vect{x}_0 \wedge \phi(x,y).$$ 

Now apply Fact~\ref{fact_ChaPil} with $\tilde{\phi}(z,y)$. We get a formula $\theta_{\tilde{\phi}}(y)$ such that for any $\aleph_0$-saturated model $\NN$ of $T$ and $b\in \NN$ we have that $\NN\models \theta_{\tilde{\phi}}(b)$ if and only if there exist tuples $a$ and $c$ in $\NN$ such that $\phi(a,b)$ holds, $c = \vect{a}_0$ and for all $i$, $c_i\notin \acl_T(b)$. Equivalently $\NN\models \theta_{\tilde{\phi}}(b)$ if and only if there exists a tuple $a$ from $\NN$  such that $a$ is $\F_q$-linearly independent over $\acl_T(b)$ and $\NN\models \phi(a,b)$. 
Using \ref{EXT} of $\indi a$, it is an easy exercise to prove that this condition is equivalent to $(H_4)$, hence the triple $(T,T_0,\LL_0)$ is suitable.
\end{proof}

\begin{lm}\label{lm_ax2}
  Let $\psi(x,y)$ be an $\LL_V$-formula. Assume that in a saturated model $(\MM,V)$ of $T_V$ the following holds for some tuple $b$ from $\MM$, for all $\LL$-formula $\phi(x,y)$: 
  $$\theta_\phi(b)\rightarrow \exists x \phi(x,b)\wedge \psi(x,b).$$
  Then for all $\phi(x,y)$, if $\MM\models \theta_\phi(b)$ then there exists a realisation $a$ of $\phi(x,b)\wedge \psi(x,b)$ such that $a$ is linearly independent over $\acl_T(b)$.
\end{lm}
\begin{proof}
  Let $\Sigma(x,y)$ be the partial type expressing ``$x$ is linearly independent over $\acl_T(y)$''. We claim that $\set{\phi(x,b)\wedge \psi(x,b)}\cup \Sigma(x,b)$ is consistent. Indeed, let $\Lambda(x,b)$ be a finite conjunction of formulae in $\Sigma(x,b)$. As $\theta_\phi(b)$ holds, there exists a realisation $a$ of $\phi(x,b)$ which is $\F_q$-linearly independent over $\acl_T(b)$, hence in particular $a$ satisfies $\phi(x,b)\wedge \Lambda(x,b)$, hence $\MM\models \theta_{\phi\wedge\Lambda}(b)$. By hypothesis, the formula $\phi(x,b)\wedge \Lambda(x,b)\wedge \psi(x,b)$ is consistent, hence we conclude by compactness.
\end{proof}

\begin{prop}[Axioms for $TV$]\label{prop_axV}
The theory $TV$ is axiomatised by adding to $T_V$ the following $\LL_V$-sentences, for all tuples of variable $y_V\subset y$, $x_V \subset x$ and $\LL$-formula $\phi(x,y)$
\[
  \forall y (\vect{y}\cap V = \vect{y_V} \wedge \theta_{\phi}(y)) \rightarrow (\exists x \phi(x,y)\wedge \vect{xy}\cap V = \vect{x_V y_V}).\tag{$A_1$}
\]
Equivalently, the theory $TV$ is axiomatised by adding to $T_V$ the following $\LL_V$-sentences, for all tuples of variable $y^1\subseteq y$, $x_V \subset x$ and $\LL$-formula $\phi(x,y)$
\[
  \forall y (\vect{y^1}\cap V = \set{0} \wedge \theta_{\phi}(y)) \rightarrow (\exists x \phi(x,y)\wedge \vect{xy^1}\cap V = \vect{x_V}).\tag{$A_2$}
\]
\end{prop}
\begin{proof}
  It is clear that the system of axioms $(A_1)$ is equivalent to the one given in Theorem~\ref{model_com_gen}. It is also clear that the system of axioms $(A_1)$ implies the system of axioms $(A_2)$. We show that the two systems are equivalent. 
  Assume that the system $(A_2)$ is satisfied in an $\aleph_0$ saturated model $(\MM,V)$ of $T_V$. Let $\phi(x,y)$ be given, and subtuples $y_V$ of $y$ and $x_V$ of $x$. We show that $(\MM,V)$ satisfies the axiom of the form $(A_1)$ given by $y_V\subset y$, $x_V \subset x$ and $\phi(x,y)$. 
  Assume that for some tuple $b$ from $\MM$, the formula $\vect{b}\cap V = \vect{b_V}\wedge \theta_\phi(b)$ holds. Let $b^1$ be a subtuple of $b$ which is a basis of $\vect{b}$ over $\vect{b_V}$. We have $\vect{b^1}\cap V = \set{0}$ hence using an instance of an axiom $(A_2)$, there exists a realisation $a$ of $\phi(x,b)$ such that $\vect{ab^1}\cap V = \vect{a_V}$. Since $b_V\subseteq V$, it follows from \ref{BMON} that $\vect{ab}\cap V = \vect{a_V b_V}$.
\end{proof}

\begin{lm}\label{lm_TV_inf}
Assume that $T$ is model complete and eliminates the quantifier $\exists^\infty$. Then $TV$ eliminates the quantifier $\exists^{\infty}$, so $(TV,T_0,\LL_0)$ is also a suitable triple.
\end{lm}

\begin{proof}
  Assume that $\abs{x} = 1$. From the description of types (see Proposition~\ref{cor_com}), types in $TS$ are obtained by adding to the types in $T$ the description of $V$ on the algebraic closure. By compactness, every $\LL_V$-formula $\phi(x,y)$ is equivalent to a disjunction of formulae of the form 
$$\exists z  \psi(x,z,y) \wedge \vect{xz}\cap V = \vect{z_V}$$
where $\psi(x,z,y)$ is an $\LL$-formula (not necessarily quantifier-free) and $z_V$ a subtuple of variables of $z$\footnote{Actually we might assume that every realisation of $z$ in $\psi$ is algebraic over the realisations of $x,y$ in $\psi$, but we don't need this fact here. Also, we may replace the condition $\vect{xz}\cap V = \vect{z_V}$ by $\vect{z}\cap V = \vect{z_V}$, but we assume that the formula gives a description of $V$ on $\vect{xz}$ in order to simplify the proof.}. In order to prove elimination of $\exists^{\infty}$, by the pigeonhole principle , we may assume that $\phi(x,y)$ is equivalent to such a formula. Now let $u,v$ be two tuples of variables such that $\abs{u}+\abs{v}\leq \abs{z}+1$, and let $u_V\subset u$, $v_V \subset v$ be two subtuples. Let $\Gamma^{uv}_{u_Vv_V}(u, yv)$ be the following $\LL$-formula
$$ \exists xz \psi(x,z,y)\wedge \vect{xz} = \vect{uv} \wedge \vect{z_v} = \vect{u_V v_V}\wedge x\in \vect{uv}\setminus \vect{v}.$$
Let $\Lambda(y)$ be the formula 
$$\bigvee_{\abs{uv}\leq \abs{z}+1, u_V\subseteq u,v_V\subseteq v, \abs{u}\geq 1} \exists v (\vect{v}\cap V = \vect{v_V} \wedge \theta_{\Gamma^{uv}_{u_Vv_V}} (yv)).$$
\emph{Claim: } For all tuple $b$ from a saturated model $(\MM, V)$ of $TV$, $(\MM,V )\models \Lambda(b)$ if and only if there exists $a\in \MM$ such that $(\MM,V)\models \phi(a,b)$ and $a\notin \acl_T(b)$.

From left to right. If $\Lambda(b)$ holds for some $b$, there exists a formula $\Gamma = \Gamma_{u_Vv_V}^{uv}$ and some tuple $e$ from $\MM$ and a subtuple $e_V$ of $e$ such that $V\cap \vect{e} = \vect{e_V}$ and $\MM \models \theta_{\Gamma}(be)$. Using one instance of the axioms $(A_1)$ (Proposition~\ref{prop_axV}) and Lemma~\ref{lm_ax2}, there exists a realisation $d$ of $\Gamma(u,be)$ such that $\vect{dbe}\cap V = \vect{d_Vb_Ve_V}$, for $d_V$ the subtuple associated to the variables $u_V$ and such that $d$ is linearly independent over $\acl_T(be)$. Using that $d$ is linearly independent over $\vect{de}$, we obtain that $\vect{de}\cap V = \vect{d_Ve_V}$. As $(\MM, V)\models \Gamma(d,be)$, there exists $a$ and a tuple $c$ from $\MM$ such that 
\begin{itemize}
\item $\MM \models \psi(a, c,b)$
\item $\vect{ac} = \vect{de}$
\item $\vect{c_V} = \vect{d_V e_V}$
\item $a\in \vect{de}\setminus \vect{e}$.
\end{itemize}
Now as $\vect{de}\cap V = \vect{d_Ve_V}$ we have $\vect{ac}\cap V = \vect{c_V}$ so $(\MM,V)\models \phi(a,b)$. Now as $d$ is linearly independent over $\acl_T(be)$ and $a\in \vect{de}\setminus \vect{e}$ we have $a\notin \acl_T(be)$ so $a\notin \acl_T(b)$. 

From right to left. Assume that $(\MM,V)\models \phi(a,b)$ and $a\notin \acl_T(b)$. Let $c$ be such that $c\models \psi(a,z,b)$ and $\vect{ac}\cap V = \vect{c_V}$. Let $e_V$ be a basis of $\acl_T(b)\cap V \cap \vect{ac}$, and complete it in a basis $e$ of $\acl_T(b)\cap \vect{ac}$. Let $d_V$ be a basis of a complement of $\vect{e_V}$ inside $\vect{ac}\cap V$ and complete it in a basis $d$ of a complement of $\vect{ed_V}$ inside $\vect{ac}$. As $a\in \vect{de}\setminus \acl_T(b)$ we have $a\in \vect{de}\setminus \vect{e}$. It is clear that $(\MM,V)\models \Gamma_{u_Vv_V}^{uv}(d,be)$ for the appropriate choice of subtuple of variables $u_V\subseteq u$ and $v_V\subseteq v$. Furthermore, as $d$ is linearly independent over $\acl_T(b) = \acl_T(be)$, we have $\theta_\Gamma(be)$, and so $\Lambda(b)$ holds.
\end{proof}

\begin{cor}\label{cor_evgen}
  Assume that $T$ is model-complete and eliminates $\exists^{\infty}$. Let $T_{V_1\dots V_n}$ be the theory whose models are models of $T$ in which $V_i$ is a predicate for a vector subspace over $\F_q$. Then $T_{V_1\dots V_n}$ admits a model companion $TV_1\dots V_n$.
\end{cor}

\begin{proof}
  This is an immediate consequence of Lemma~\ref{lm_TV_inf} and Proposition~\ref{prop_iterate}. 
\end{proof}

\begin{ex}[Generic vector subspace of a vector space]\label{ex_vspace}
Consider the theory $T$ of infinite $\F_q$-vector spaces in the language $\LL = \set{(\lambda_\alpha)_{\alpha\in \F_q}, +, 0 }$. Applying Corollary~\ref{cor_evgen} the theory $T_{V_1\dots V_n}$ admits a model companion $TV_1\dots V_n$. Proposition~\ref{prop_stable} gives us inductively that $TV_1\cdots V_n$ is stable for all $n\in \N$. It is easy to check that $TV_1$ is the theory of belles paires (see~\cite{Po83}) of the theory $T$. 
\end{ex}

\subsection{Field of positive characteristic with generic vector subspaces}

Let $p>0$ be a prime number.
Let $\LL = \set{+,-,\cdot,0,1,\dots}$ and $T$ an $\LL$-theory of an infinite field of characteristic $p$. Let $\F_{q_1},\cdots,\F_{q_n}$ be finite subfields in any model of $T$. Consider the theory $T'$ obtained by adding to the language a constant symbol for each element of $\F_{q_1}\cup \cdots \cup \F_{q_n}$. Then $T$ and $T'$ have the same models. It follows that for each $i$ we may consider that the theory of infinite $\F_{q_i}$-vector space in the language $\LL_i = \set{+,0,(\lambda_\alpha)_{\alpha\in \F_{q_i}}}$ is a reduct of $T$.

\begin{prop}\label{prop_gensubvectfields}
  Let $\LL \supseteq \LLr$ and $T$ an $\LL$-theory of an infinite field of characteristic $p$. Let $\F_{q_1},\cdots,\F_{q_n}$ be finite subfields in any model of $T$. Assume that
  \begin{enumerate}
    \item $T$ is model-complete;
    \item $T$ eliminates $\exists^\infty$. 
  \end{enumerate}
  Let $T_{V_1\dots V_n}$ be the theory whose models are models of $T$ in which each $V_i$ is a predicate for an $\F_{q_i}$-vector subspace. By Corollary~\ref{cor_evgen} the theory $T_{V_1\dots V_n}$ admits a model-companion.
\end{prop}

An additive subgroup of a field of characteristic $p$ is an $\F_p$-vector space, hence Proposition~\ref{prop_gensubvectfields} translates as follows.

\begin{prop}\label{prop_genfield}
  Let $\LL \supseteq \LLr$ and $T$ an $\LL$-theory of an infinite field of characteristic $p$. Assume that
  \begin{enumerate}
    \item $T$ is model-complete;
    \item $T$ eliminates $\exists^\infty$. 
  \end{enumerate}
  Let $T_{G_1\dots G_n}$ be the theory whose models are models of $T$ in which each $G_i$ is a predicate for an additive subgroup. By Corollary~\ref{cor_evgen} the theory $T_{G_1\dots G_n}$ admits a model-companion, which we denote by $TG_1\dots G_n$.
\end{prop}

\begin{ex}\label{ex_fields}
  The hypotheses of Propositions~\ref{prop_gensubvectfields} and \ref{prop_genfield} are satisfied by the following theories:
\begin{itemize}
  \item $\ACF_p$, $\SCF_{p,e}$ for $e$ finite or infinite, $\Psf_c$,
\item $\ACFA_p$, $\DCF_p$.
\end{itemize}
All these theories are model-complete. Concerning the elimination of $\exists^{\infty}$, all perfect $\PAC$ fields are \emph{algebraically bounded}~\cite{ChaHru04}, which implies the elimination of $\exists^{\infty}$ for $\ACF_p$ and $\Psf_p$. Elimination of $\exists^{\infty}$ for $\SCF_{p,e}$ follows from~\cite[Proposition 61.]{Del88}. The theory $\ACFA_p$ eliminates $\exists^\infty$ in all characteristic, this follows easily from the definability of the $\sigma$-degree (see~\cite[Section 7]{CH99}). For all $p$ prime or $0$, the theory $\DCF_p$ eliminates the quantifier $\exists^\infty$, this follows from the proof of this result in ~\cite[Theorem 2.13, p51]{MMP96}, although it was proved in the characteristic $0$ case, the proof works in all characteristics.
\end{ex}

\begin{ex}[$\ACF\mathrm{V_1\cdots V}_n$ and $\ACFG$]\label{ex_ACFG}
  Let $\F_{q_1},\cdots,\F_{q_n}$ be any finite fields of characteristic $p$. We denote by $\ACF\mathrm{V_1\cdots V}_n$ and $\ACFG$ respectively the theories $\ACF_p\mathrm{V_1\cdots V}_n$ and $\ACF_p G$. \cite{dE18B} is dedicated to a detailed study of the theory $\ACFG$, which is $\NSOP{1}$ and not simple (see also Example~\ref{ex_ACFG_NSOP}).
\end{ex}

Recall that a pseudo-algebraically closed field is a field $K$ which is existentially closed in every regular extension. The theory $\PAC$ is incomplete but eliminates $\exists^\infty$ if the field is perfect. 
\begin{prop}\label{prop_PACG}
  Let $\PAC_{pG}$ be the theory whose models are perfect $\PAC_p$-fields in $\LLr$ with a predicate $G$ for an additive subgroup. Then there exists a theory $\PAC_pG$ such that 
  \begin{enumerate}
    \item every model $(F,G')$ of $\PAC_{pG}$ extends to a model $(K,G)$ of $\PAC_pG$ such that $K$ is a regular extension of $F$;
    \item every model $(K,G)$ of $\PAC_pG$ is existentially closed in every extension $(F,G')$ such that $F$ is a regular extension of $K$.
  \end{enumerate}
  Let $T$ be a theory of perfect $\PAC_p$-fields in a language containing $\LLr$ such that $T$ is model-complete, and $T_{G_1\cdots G_n}$ be the theory whose models are models of $T$ with predicates $G_i$ for additive subgroups. Then $T_{G_1\cdots G_n}$ admits a model-companion, $TG_1 \cdots G_n$.
\end{prop}
\begin{proof}
  Perfect $\PAC_p$-fields in $\LLr$ satisfies $(H_4)$, the proof of this in Theorem~\ref{thm_gensubvect} does not use the model-completeness of the theory $T$, so the first statement follows from Proposition~\ref{cor_com_gen}. The second statement is Corollary~\ref{cor_evgen}.
\end{proof}
\begin{rk}
  Note that the \emph{perfect} assumption is only here to ensure that the fields eliminate the quantifier $\exists^\infty$. It should be true that all $\PAC$ fields eliminate the quantifier $\exists^\infty$ although we did not find any reference in the literature. 
\end{rk}

\subsection{$\NSOP 1$ fields of positive characteristic with generic vector subspaces}\label{sec_NSOPf}

Now we give some condition under which the theory obtained in Proposition~\ref{prop_gensubvectfields} is $\NSOP 1$.
In this subsection, for $A$ in some field, we denote by $\acl_T$ the model-theoretic algebraic closure, $A^{s}$ the separable closure and $\ol{A}$ the field theoretic algebraic closure. For three subsets $A,B,C$ of some field, the notation $A\indi{ld}_C B$ stands for “the field spanned by $A$ and $B$ are linearly disjoint over the field spanned by $C$”. Recall that a field extension $A\subseteq B$ (which we will denote $B/A$) is \emph{regular} if $B\indi{ld}_A \ol A$.

We denote by $\LLr=\set{+,-,\cdot,0,1}$ the language of rings. By an \emph{arbitrary} theory of fields $T$, we mean a theory of field in a language $\LL\supseteq \LLr$.
\begin{fact}\label{fact_regacl}
  \begin{enumerate}
    \item Let $T$ be an arbitrary theory of fields. Let $F\models T$ and $A\subseteq F$. Then $F/acl_T(A)$ is a regular extension~\cite[(1.17)]{C99}.
    \item    Let $E\subset K\cap L$ be three fields. If $K/E$, $L/E$ are regular and $K\indi{ld}_E L$ then $K^s\indi{ld}_{E^s} L^s$ \cite[Lemma 3.1 (1)]{C99}.
    \item  Assume that $A,B,C$ are separably closed subfields of a separably closed field $F$ such that $C\subseteq A\cap B$. If $A\indi{ld}_C B$ and $F/AB$ is separable, then $tp^{SCF_{p,e}}(A/B)$ does not fork over $C$~\cite[Remark after (1.2)]{C02}.
\end{enumerate}
\end{fact}

The following gives a behaviour of the Kim-independence in \emph{any} theory of fields.
\begin{fact}[{\cite[Proposition 9.28]{KR17}}, {\cite[Theorem 3.5]{C99}}]\label{fact_kimfield}
  Let $T$ be an arbitrary theory of fields, and $E\prec F\models T$. Let $A,B$ be $\acl_T$-closed subsets of $F$ containing $E$, such that $A\indi K _E B$. Then
  \begin{enumerate}
    \item $A\indi{ld}_E B$;
    \item $F/AB$ is a separable extension;
    \item $\acl_T(AB)\cap A^sB^s = AB$.
    \end{enumerate}
\end{fact}

\begin{lm}\label{lm_CH}
  Let $T$ be an arbitrary theory of fields, and $F\models T$. Let $A,B,C,D$ be subsets of $F$, containing some set $k\subseteq F$, and such that $A,B\subseteq D$. Assume that $A$ and $B$ are $\acl_T$-closed and that $tp^T(D/C)$ is finitely satisfiable in $k$. Then we have the following results.
  \begin{enumerate}
    \item $(F\cap (AC)^s+F\cap (BC)^s)\cap D = A+B$;
    \item $[(F\cap \ol{AC})\cdot (F\cap \ol{BC}))] \cap D = A\cdot B$ (where $U\cdot V= \set{u\cdot v \mid u\in U,\ v\in V}$).
    \end{enumerate}
\end{lm}
\begin{proof}
  We prove \textit{(1)}, the other is proved by a similar argument. Let $v_1\in F\cap (AC)^s$, $v_2\in F\cap (BC)^s$ and $u\in D$ be such that $u = v_1+v_2$. There exist nontrivial separable polynomials $P(X,a,c)$ and $Q(X,b,c')$ with leading coefficients $1$ such that $v_1$ is a root of $P(X,a,c)$ and $v_2$ is a root of $Q(X,b,c')$, $a$ a tuple in $A$, $b$ a tuple in $B$. The formula $\phi(z_1,z_2,z_3,c,c')$
  $$ \exists x\exists y \ x+y = z_1\wedge P(x,z_2,c) = 0\wedge Q(y,z_3, c') = 0$$
  is in $tp^{T}(u,a,b/C)$, which is finitely satisfiable in $k$. Hence, there exists $d,d'\in k$ such that $\phi(z_1,z_2,z_3,d,d')\in tp^{T}(u,a,b/k)$ and so $u\in A+B$ as $A$ and $B$ are $\acl_T$-closed.
\end{proof}

\begin{lm}\label{lm_sepinter}
Let $A,B$ be two extensions of some field $E$, such that $AB/E$ is regular and $A\indi{ld}_E B$. Then $(A^s + B^s)\cap AB = A+B$.
\end{lm}
\begin{proof}
  First, observe that $A^sB\cap B^s = E^s B$. Indeed $A/E$ and $B/E$ are regular so by Fact~\ref{fact_regacl} \textit{(2)}, we have that $A^s\indi{ld}_{E^s} B^s$ hence $A^sB \indi{ld} _{E^sB} B^s$ and so $A^sB\cap B^s = E^s B$. Symmetrically, we have $AB^s \cap A^s = E^s A$. If $v\in AB$ is such that $v = \alpha+\beta$ for $\alpha\in A^s$ and $\beta\in B^s$, then $\alpha = v-\beta\in AB^s \cap A^s = E^s A$. Similarly $\beta\in E^s B$. Let $L$ be a finite extension of $E$ inside $E^s$ such that $\alpha\in AL$ and $\beta\in BL$. We can complete $\set{1}$ to a basis $\set{1,u_2,\dots,u_n}$ of the $E$-vector space $L$. As $AB\indi{ld}_E L$, it is also a basis of the $AB$-vector space $LAB$. As $AB\indi{ld}_A LA$ and $AB\indi{ld}_B LB $, it is also a basis of the $A$-vector space $LA$ and of the $B$-vector space $LB$. Now the coordinates of $v\in AB$ in the $AB$-vector space $LAB$ are $(v,0,\dots,0)$ as $v = v+0u_2+\dots+0u_n$. Let $(a_1,\dots,a_n)$ (respectively $(b_1,\dots,b_n)$) be the coordinates of $\alpha$ with respect to the basis $(1,u_2,\dots,u_n)$ of the $A$-vector space $LA$ (respectively of $\beta$ in this basis of the $B$-vector space $LB$). As $v = \alpha+\beta$, we have, looking at the first coordinate that $v = a_1+b_1$, so $v\in A+B$.
\end{proof}

\begin{thm}\label{thm_NSOPfields}
  Let $T$ be an arbitrary theory of fields which is model-complete, $\NSOP{1}$, and eliminates $\exists^{\infty}$. Let $\F_{q_1},\cdots,\F_{q_n}$ be subfields. Assume that $T$ satisfies the following assumption for all $\acl_T$-closed $A,B$ and $E\models T$ contained in $A$ and $B$:
  \begin{center}if $A \indi T _E B$ then $\acl_T(AB)\subseteq \ol{AB}$.\end{center}

Then $TV_1\dots V_n$ is $\NSOP{1}$ and Kim-independence in $TV_1\dots V_n$ is given by
$$A\indi w _E B \iff A\indi T _E B \text{ and for all $i\leq n$ } V_i(A+B) = V_i(A)+V_i(B)$$
(for $A,B,C$ $\acl_T$-closed, $A,B$ containing $E$, $E\models T$).
\end{thm}

\begin{proof}
  We prove that $\indi T$ satisfies the conditions of Corollary~\ref{cor_iterateNSOP}. Let $\indi{i}$ the independence in the sense of $\F_{q_i}$-vector space, we want to show that for all $i=1,\ldots,n$,
\begin{itemize}
  \item[$(A)$] for all model $E$ of $T$ and $A,B,C$ algebraically closed containing $E$, if $C\indi T _E A,B$ and $A\indi T _E B$ then 
  $$(\acl_T(AC),\acl_T(BC)) \indi{i}_{A,B} \acl_T(AB).$$
\end{itemize}
Let $F\models T$, let $E\prec F$ and $A,B,C$ in $F$ containing $E$, with $C\indi T _E A,B$ and $A\indi T _E B$. For all $i=1,\cdots,n$, the condition $(\acl_T(AC),\acl_T(BC)) \indi{i}_{A,B} \acl_T(AB)$ is equivalent to $$(\acl_T(AC)+\acl_T(BC))\cap \acl_T(AB)=A+B.$$
From Fact~\ref{fact_kimfield} \textit{(2)}, $F/AB$, $F/BC$ and $F/AC$ are separable extension. By our assumptions on $T$ and $A,B$ and $C$ we have that $\acl_T(AB)\subseteq (AB)^s$, $\acl_T(AC)\subseteq (AC)^s$ and $\acl_T(BC)\subseteq (BC)^s$, so 
$$(\acl_T(AC)+\acl_T(BC))\cap \acl_T(AB)\subseteq ((AC)^s+(BC)^s)\cap (AB)^s.$$

\begin{claim} $((AC)^s+(BC)^s)\cap (AB)^s = A^s + B^s$\end{claim}
\textit{Proof of the claim.} First, observe that as fields, $E^s$ is an elementary substructure of $F^s$. Indeed, by model completeness of $Th(E^s)$ (which is $\mathrm{SCF}_{p,e}$ for some $e\leq \infty$) we have to check that they have the same imperfection degree (which is clear as $F\succ E$) and that $F^s/E^s$ is separable (the later follows from the fact that $F/E$ is a regular extension). Now by Fact~\ref{fact_kimfield} \textit{(1)} we have $C\indi{ld}_E AB$. As $E$ is a model, $C/E$ and $AB/E$ are regular extensions\footnote{In fact here we only use that $E = \acl_T(E)$, and Fact \ref{fact_regacl} \textit{(1)}.}, by Fact~\ref{fact_regacl} \textit{(2)} we have that 
$$C^s\indi{ld}_{E^s} \ (AB)^s. \quad (*)$$ 
Moreover $F^s/ABC$ is separable, (as so are $F^s/F$ and $F/ABC$, the latter using ~Fact~\ref{fact_kimfield} \textit{(2)}) and so is $C^s(AB)^s/ABC$. It follows that the following extension is separable
$$F^s/C^s(AB)^s. \quad (**)$$
From $(*)$ and $(**)$, using Fact~\ref{fact_regacl} \textit{(3)} we have that $tp_{\SCF}(C^s/(AB)^s)$ does not fork over $E^s$. By stability, as $E^s$ is an elementary submodel of the ambiant model $F^s$ of $\SCF_{p,e}$, $tp_{\SCF}(C^s/(AB)^s)$ is a coheir of $tp_{\SCF}(C^s/E^s)$. 
From Lemma~\ref{lm_CH}, it follows that $((AC)^s+(BC)^s)\cap (AB)^s = A^s + B^s$.

By the claim $(\acl_T(AC)+\acl_T(BC))\cap \acl_T(AB)\subseteq (A^s+B^s)\cap \acl_T(AB)$. Now by Fact~\ref{fact_kimfield} \textit{(3)}, we have that $A^sB^s\cap \acl_T(AB) = AB$ so $(A^s+B^s)\cap \acl_T(AB) \subseteq (A^s+B^s)\cap AB$. 
Finally, by Lemma~\ref{lm_sepinter}, as $AB/E$ is regular and $A\indi{ld}_E B$, we have $(A^s+B^s)\cap AB = A+B$. 
\end{proof}

I am grateful to Michel Matignon for pointing out to me this nice proof of the following simple fact.
\begin{lm}\label{lm_sumprodgen}
  Let $K$ be a field and $K(X,Y)$ be a rational function field in two variables. Then
  $$XY\notin K(X)+ K(Y);$$
  $$X+ Y \notin K(X)\cdot K(Y);$$
  where $K(X)\cdot K(Y) = \set{uv \mid u\in K(X),v\in K(Y)}$.
\end{lm}

\begin{proof}
  There exists a derivative $D:K(X,Y)\rightarrow K(X,Y)$ such that $D(K(Y)) = \set{0}$ and $D$ extends the canonical derivation on $K(X)$ (namely the \emph{partial} derivative with respect to $X$, see~\cite[Proposition 23.11]{Mor96}). Let $u\in K(X)$ and $v\in K(Y)$. If $XY = u+v$ then applying $D$ we get $Y = Du\in K(X)$ a contradiction. If $X+Y = uv$ then applying $D$ we get $1 = v Du$ hence, as $Du\in K(X)$, $v\in K(X)\cap K(Y) = K$. Now $Y = uv-X\in K(X)$ a contradiction.
\end{proof}

\begin{prop}\label{prop_fieldsnotsimple}
  Let $T$ be an arbitrary theory of fields satisfying the same hypotheses as Theorem~\ref{thm_NSOPfields}. Then $TV_1\cdots V_n$ is not simple.
\end{prop}

\begin{proof}
  To prove that $TV_1\cdots V_n$ is not simple, it is sufficient to prove that $TV$ is not simple. Let $E\prec F$ be models of $T$ and $a,b,d$ elements of $F$ be such that $a\indi T_E b,d$ and $b\indi T _E d$. We show that $$ad+b\in \left[\acl_T(Ead)+\acl_T(Ebd)\right]\setminus \left[(\acl_T(Ea)+\acl_T(Ebd))\cup \acl_T(Ead)\right],$$ then $TV$ is not simple, by Corollary~\ref{cor_notsimple}. Since $b\notin \acl_T(Ead)$, it is clear that $ad+b\notin \acl_T(Ead)$. Assume that $ad+b\in \acl_T(Ea)+\acl_T(Ebd)$. Then $ad\in \acl_T(Ea)+\acl_T(Ebd)$, let $u\in \acl_T(Ea)$ and $v\in \acl_T(Ebd)$ be such that $ad = u+v$. From Fact~\ref{fact_kimfield}, we have that $\acl_T(Ea)\indi{ld}_E \acl_T(Ebd)$, hence $\acl_T(Ea)(d)\indi{ld}_{E(d)} \acl_T(Ebd)$ so $\acl_T(Ea)(d)\cap \acl_T(Ebd) = E(d)$. Similarly, $\acl_T(Ebd)(a)\cap \acl_T(Ea) = E(a)$. It follows that 
  \begin{align*}
    u &= ad-v \in \acl_T(Ebd)(a)\cap \acl_T(Ea) = E(a) \\
    v &= ad-u\in \acl_T(Ea)(d)\cap \acl_T(Ebd) = E(d) 
  \end{align*}
hence $ad\in E(a)+E(d)$, which contradicts Lemma~\ref{lm_sumprodgen}.
\end{proof}

\begin{ex}[The theories $\mathrm{ACFV_1\dots V}_n$ and $\ACFG$]\label{ex_ACFG_NSOP}
Let $\mathrm{ACFV_1\dots V}_n$ and $\ACFG$ be the theories as in Example~\ref{ex_ACFG}. By Theorem~\ref{thm_NSOPfields} and Proposition~\ref{prop_fieldsnotsimple} those theories are $\NSOP{1}$ not simple. In $\mathrm{ACFV_1\dots V}_n$, Kim-independence agrees with the relation 
  $$A\indi w _C B\iff A\indi{\ACF}_C\ \ B\text{ and for all $i\leq n$, } {V_i}(\overline{AC}+ \overline{BC}) = V_i(\overline{AC})+ V_i(\overline{BC}).$$
Furthermore, $\indi w$ satisfies 
\begin{itemize}
\item \bsc{Strong Finite Character over algebraically closed sets.} For algebraically closed $E$, if $a\nindi w _E b$, then there is a formula $\phi(x,b,e)\in tp^{ACFV_1\dots V_n}(a/bE)$ such that for all $a'$, if $a'\models \phi(x,b,e)$ then $a'\nindi w _E b$.
\item \bsc{$\indi a$-amalgamation over algebraically closed sets}. For algebraically closed set $E$ if there exists tuples $c_1,c_2$ and sets $A,B$ such that
\begin{itemize}
\item $c_1\equiv^{ACFV_1\dots V_n}_E c_2$
\item $\overline{AE}\cap \overline{BE} = E$
\item $c_1\indi w_E A$ and $c_2\indi w_E B$ 
\end{itemize}
then there exists $c\indi w_E A,B$ such that $c\equiv^{ACFV_1\dots V_n}_A c_1$, $c \equiv^{ACFV_1\dots V_n}_B c_2$, $A\indi a _{Ec} B$, $c\indi a _{EA} B$ and $c\indi a _{EB} A$.
\end{itemize}
This is Theorem~\ref{thm_conserve}, knowing that $\indi{\ACF}\quad $ is stationary over algebraically closed sets hence satisfies the independence theorem over algebraically closed sets without any assumption on the parameters.
\end{ex}

\begin{ex}\label{ex_NSOPfield}
  Perfect $\omega$-free $\PAC_p$ fields are $\NSOP 1$, furthermore, as they are algebraically bounded, the condition on the algebraic closure in Theorem~\ref{thm_NSOPfields} is satisfied. If $T$ is a theory of a perfect $\omega$-free $\PAC_p$-field in an expansion of the language $\LLr$ such that $T$ is model-complete, then $TG_1 \cdots G_n$ (Proposition~\ref{prop_PACG}) is $\NSOP 1$. This holds of course for any $\NSOP 1$ perfect $\PAC_p$ field.
\end{ex}

\subsection{Algebraically closed fields with a generic multiplicative subgroup}\label{sec_genmult}

We are now interested in using Theorem~\ref{model_com_gen} to prove that the theory of algebraically closed fields of fixed arbitrary characteristic with 
a predicate for a multiplicative subgroup admits a model companion. Consider $\LLf=\set{+,-,\cdot,^{-1},0,1}$ and $\LL_0 = \set{\cdot, ^{-1},1}\subseteq \LLf$.

The pure multiplicative group of any field is an $\aleph_1$-categorical abelian group, its model theory is described in \cite{Mac71}, see also \cite[Chapter VI]{Cher76}.

Fix $p$ a prime or $0$. Consider the theory $\ACF_p$. The theory $\ACF_p\upharpoonright \LL_0$ is complete and we will identify it with the theory of the multiplicative group of an algebraically closed field of characteristic $p$, denoted by $T_p$. 
The theory $T_p$ is axiomatised by adding to the theory of abelian groups the following sets of axiom:
\begin{itemize}
\item \emph{If $p>0$: }$\set{\forall x \ \exists^{=n} y \ \ y^n = x\mid n\in \N\setminus p\N}\cup \set{\forall x \exists^{=1}y\ y^p = x}$
\item \emph{If $p=0$: }$\set{\forall x \  \exists^{=n} y \ \ y^n = x\mid n\in \N\setminus \set{0}}.$
\end{itemize}

\begin{prop}\label{prop_pregeo_mult0}
The theory $T_p$ has quantifier elimination in the language $\LL_0$. It is strongly minimal hence $\aleph_1$-categorical. Furthermore for any subset $A$ of a model $M$ of $T_p$, the algebraic closure is given by $$\acl_p(A) := \set{u\in M, u^n\in \vect{A} \text{ for some $n\in \N\setminus\set{0}$}}$$
where $\vect{A}$ is the group spanned by $A$. Every $\acl_p$-closed set is a model of $T_p$. Furthermore $\acl_p$ defines a pregeometry which is modular and the associated independence relation in $T_p$ is given by 
$$A\indi{p}_C B :\iff \acl_p(AC)\cap \acl_p(BC) = \acl_p(C).$$
\end{prop}

 \begin{lm}\label{lm_var}
 Let $K\models \ACF$, $V\subset K^{n}$ an affine (irreducible) variety, $\mathcal O \subset K^n$ a Zariski open set. The following are equivalent:
 \begin{enumerate}
 \item for all $k_1,\dots,k_n \in \N$, $c\in K$ the quasi affine variety $V\cap \mathcal O$ is not included  in the zero set of $x_1^{k_1}\cdot \dots \cdot x_n^{k_n} = c$
  \item for all $k_1,\dots,k_n \in \N$, $c\in K$ the variety $V$ is not included  in the zero set of $x_1^{k_1}\cdot \dots \cdot x_n^{k_n} = c$
  \item there exist $L\succ K$ and a tuple $a$ which is multiplicatively independent over $K$ and with $a\in (V\cap \mathcal O)(L)$
 \end{enumerate}
 \end{lm}
 
 \begin{proof}
 \textit{(1)} implies \textit{(2)} is trivial. We show that \textit{(2)} implies \textit{(3)}. Assume that \textit{(3)} does not hold. Take a generic $a$ over $K$ of the variety $V$ in some $L\succ K$. We have $a\in \mathcal O$. Then there exists $k_1,\dots,k_n\in \N$ such that $a_1^{k_1}\cdot \dots \cdot a_n^{k_n} = c$ for some $c\in K$. By genericity of $a$, it follows that $V$ is included in the zero set of $x_1^{k_1}\cdot \dots \cdot x_n^{k_n} = c$, hence \textit{(2)} does not hold. \textit{(3)} implies \textit{(1)} follows easily from the fact that $V$ and $\mathcal O$ are definable over $K$.
 \end{proof}

The following fact was first observed in the proof of Theorem 1.2 in~\cite{BGH13}, it is also Corollary 3.12 in~\cite{MCT17}.
\begin{fact}\label{fact_BGH}
 Let $p$ be a prime number or $0$. Let $\phi(x,y)$ an $\LLf$-formula such that for all tuple $b$ in a model of $\ACF_p$, $\phi(x,b)$ defines an affine variety. Then there exists an $\LLf$-formula $\theta_\phi(y)$ such that for any model $K$ of $\ACF_p$ and tuple $b$ from $K$, we have $K\models \theta_\phi(b)$ if and only if 
for all $k_1,\dots,k_n \in \N$, $c\in K$, the set $\phi(K,b)$ is not included  in the zero set of $x_1^{k_1}\cdot \dots \cdot x_n^{k_n} = c$.
 \end{fact}
 
 It is standard that every definable set in $\ACF_p$ can be written as a finite union of quasi-affine varieties. Furthermore, by~\cite[Lemma 3.10]{MCT17}, given any $\LLr$-formula $\vartheta(x,z)$, the set of $c$ such that $\vartheta(x,c)$ is a quasi-affine variety is a definable set. Let $\CCCC$ be the class of formulae $\vartheta(x,z)$ such that for all $K\models \ACF_p$ and $c$ tuple from $K$, the set $\vartheta(K,c)$ is a quasi-affine variety. 
 
 \begin{lm}\label{lm_cm}
   Let $p$ be a prime number or $0$. For any $\vartheta(x,z)\in \CCCC$ there exists an $\LLf$-formula $\theta_\vartheta(z)$ such that for any model $K$ of $\ACF_p$ and tuple $c$ from $K$, we have $K\models \theta_\vartheta(c)$ if and only if there exists $a$ such that $\models \vartheta(a,c)$ and $a$ is $\indi p$-independent over $K$.
\end{lm}
\begin{proof}
  Let $K\models \ACF_p$ and $\vartheta(x,z)\in \CCCC$. Using~\cite[Theorem 10.2.1]{Joh16}, there exists a formula $\tilde\vartheta(x,z)$ such that for all tuple $c$ from $K$, the set $\tilde\vartheta(K,c)$ is the Zariski closure of $\vartheta(K,c)$. Now by Fact~\ref{fact_BGH}, there exists a formula $\theta(z)$ such that $K\models \theta(c)$ if and only if $\tilde\vartheta(K,c)$ is not included  in the zero set of $x_1^{k_1}\cdot \dots \cdot x_n^{k_n} = d$, for all $d\in K$, $k_1,\cdots,k_n\in \N$. By Lemma~\ref{lm_var}, $K\models \theta(c)$ if and only if there exist $L\succ K$ and a tuple $a$ which is multiplicatively independent over $K$ and with $a\models \vartheta(x,c)$.
\end{proof}

If $G^\times$ is a symbol for a unary predicate, we denote by $\ACF_{G^\times}$ the theory in the language $\LLr\cup \set{G^\times}$ whose models are algebraically closed fields of characteristic $p$ in which the predicate $G^\times$ consists of a multiplicative subgroup.

\begin{thm}\label{thm_genmult}
The theory $\ACF_{G^\times}$ admits a model companion, which we denote by $\ACFG^\times$.
\end{thm}

\begin{proof}
We check the conditions of Definition~\ref{def_suitabletriple}
 \begin{enumerate}[label={$(H_{\arabic*})$}]
\item $\ACF_p$ is model complete;
\item $T_p$ is model-complete and for all infinite $A$, $\acl_p(A) \models T_p$;
\item[$(H_3^+)$] $\acl_p$ defines a modular pregeometry;
\item[$(H_4)$] for all $\LLf$-formula $\phi(x,y)$ there exists an $\LLf$-formula $\theta_\phi(y)$ such that for $b\in K\models \ACF_p$
\begin{eqnarray*}
\MM \models \theta_\phi(b) &\iff& \text{there exists  $L\succ K$ and $a\in L$ such that}\\
&\mbox{ }& \text{$\phi(a,b)$ and $a$ is $\indi p$-independent over $K$.}
 \end{eqnarray*}
\end{enumerate}
$\ACF_p$ is model complete by quantifier elimination. Conditions $(H_2)$ and $(H_3^+)$ follow from Proposition~\ref{prop_pregeo_mult0}. 
We don't have condition $(H_4)$ for all formulae, but only for the formulae in $\CCCC$ (Lemma~\ref{lm_cm}), which is enough for the existence of the model-companion by Remark~\ref{rk_mcvar}.
\end{proof}

Let $\ACFG^{\times}$ be the theory obtained in Theorem~\ref{thm_genmult}. We denote by $A\cdot B$ the product set $\set{a\cdot b \mid a\in A,\ b\in B}$.
\begin{thm}
Any completion of $\ACFG^\times$ is $\NSOP{1}$ and not simple. Furthermore, Kim-independence coincide over models with the relation
$$A\indi w _C B\iff A\indi{\ACF}_C\ \ B\text{ and } {G^\times}(\overline{AC}\cdot \overline{BC}) = G^\times(\overline{AC})\cdot G^\times(\overline{BC}).$$
Furthermore, $\indi w$ satisfies 
\begin{itemize}
\item \bsc{Strong Finite Character over algebraically closed sets.} For algebraically closed $E$, if $a\nindi w _E b$, then there is a formula $\phi(x,b,e)\in tp^{\ACFG^\times}(a/bE)$ such that for all $a'$, if $a'\models \phi(x,b,e)$ then $a'\nindi w _E b$.

\item \bsc{$\indi a$-amalgamation over algebraically closed sets}. For algebraically closed set $E$ if there exists tuples $c_1,c_2$ and sets $A,B$ such that
\begin{itemize}
\item $c_1\equiv^{\ACFG^\times}_E c_2$
\item $\overline{AE}\cap \overline{BE} = E$
\item $c_1\indi w_E A$ and $c_2\indi w_E B$ 
\end{itemize}
then there exists $c\indi w_E A,B$ such that $c\equiv^{\ACFG^\times}_A c_1$, $c \equiv^{\ACFG^\times}_B c_2$, $A\indi a _{Ec} B$, $c\indi a _{EA} B$ and $c\indi a _{EB} A$.
\end{itemize}
\end{thm}

\begin{proof}
Using Theorem~\ref{thm_KF}, it is enough to show that for $E$ algebraically closed and $A,B,C$ algebraically closed containing $E$, if $C\indi{\ACF} _E \ A,B$ and $A\indi{\ACF} _E \ B$ then $$\overline{AC}\cdot \overline{BC} \cap \overline{AB} = A\cdot B.$$
This easily follows from the fact that $tp^{\ACF}(C/AB))$ is finitely satisfiable in $E$, as in the proof of Theorem~\ref{thm_NSOPfields}. The rest is Theorem~\ref{thm_conserve}, knowing that $\indi{\ACF}\quad$ is stationnary over algebraically closed sets, similarly to Example~\ref{ex_ACFG_NSOP}.
To prove that $\ACFG^{\times}$ is not simple, we use Corollary~\ref{cor_notsimple}, as in the proof of Proposition~\ref{prop_fieldsnotsimple}. Let $E$ be a model of $\ACF_p$ and $a,b,d$ in an extension be such that $a\indi{\ACF}_E\quad b,d$ and $b\indi{\ACF}_E\quad d$. We claim that $$(a+d)b\in \left[\ol{Ead}\cdot \ol{Ebd}\right]\setminus \left[(\ol{Ea}\cdot \ol{Ebd})\cup \ol{Ead}\right].$$ Since $b\notin \ol{Ead}$, it is clear that $(a+d)b\notin \ol{Ead}$. Assume that $(a+d)b\in \ol{Ea}\cdot \ol{Ebd}$. Then $a+d\in \ol{Ea}\cdot \ol{Ebd}$, let $u\in \ol{Ea}$ and $v\in \ol{Ebd}$ be such that $a+d = uv$. We have that $\ol{Ea}\indi{ld}_E \ol{Ebd}$, hence $\ol{Ea}(d)\indi{ld}_{E(d)} \ol{Ebd}$ so $\ol{Ea}(d)\cap \ol{Ebd} = E(d)$. Similarly, $\ol{Ebd}(a)\cap \ol{Ea} = E(a)$. It follows that 
  \begin{align*}
    u &= (a+d)v^{-1} \in \ol{Ebd}(a)\cap \ol{Ea} = E(a) \text{ and } \\
    v &= (a+d)u^{-1}\in \ol{Ea}(d)\cap \ol{Ebd} = E(d)
  \end{align*}
hence $a+d\in E(a)\cdot E(d)$, which contradicts Lemma~\ref{lm_sumprodgen}.
\end{proof}

\subsection{Pairs of geometric structures}
Let $T$ be a pregeometric theory in a language $\LL$ with monster $\M$, $b$ a tuple from $\M$ and $\phi(x,b)$ a formula. By $\di(\phi(x,b))$ we mean the maximum dimension (in the sense of the pregeometry) of $\acl(cb)$ over $\acl(b)$, for realisations $c$ of $\phi(x,b)$.
\begin{fact}\label{fact_gagstr}
  Let $T$ be a geometric theory and $\M$ a monster model for $T$. Then for all formula $\phi(x,y)$ there exists a formula $\theta_\phi(y)$ such that $\theta_\phi(b)$ holds if and only if there exists a realisation $a$ of $\phi(x,b)$ which is an independent tuple over $\acl_T(b)$.
\end{fact}
\begin{proof}
  From~\cite[Fact 2.4]{Ga05}, for each $k\leq \abs{x}$ there exists a formula $\theta_k(y)$ such that $\theta_k(b)$ if and only if $\di(\phi(\M,b)) = k$. The formula $\theta_{\abs{x}}(y)$ holds if and only if there is a realisation $a$ of $\phi(x,b)$ such that $\di(\acl(ab)/\acl(b)) = \abs{x}$, hence $a$ is independent over $\acl_T(b)$.
  \end{proof}

Let $\LL_S$ be the expansion of $\LL$ by a unary predicate $S$. A \emph{pair} of models of $T$ is an $\LL_S$-structure $(\MM,\MM_0)$, where $\MM\models T$ and $S(\MM) = \MM_0$ is a substructure of $\MM$ model of $T$. We call $T_S$ the theory of the \emph{pairs of models of $T$}. 

\begin{prop}\label{prop_weakmc}
  Let $T$ be a model-complete geometric theory in a language $\LL$. Assume that every $\acl_T$-closed set is a model of $T$. Then there exists an $\LL_S$-theory $TS$ containing $T_S$ such that:
  \begin{enumerate}
    \item every model $(\NN,\NN_0)$ of $T_S$ has a strong extension which is a model of $TS$;
    \item every model of $TS$ is existentially closed in every strong extension model of $T_S$.
  \end{enumerate}
  Furthermore, $TS$ satisfies the conclusions of Proposition~\ref{cor_com}. Finally, if $T$ is stable, then so is $TS$.
 \end{prop}
\begin{proof}
  We check that $T,T_0,\LL_0$ satisfies the hypotheses of Theorem~\ref{model_com_gen}. $(H_1)$, $(H_2)$ and $(H_3)$ are clear, and $(H_4)$ is Fact~\ref{fact_gagstr}. The last assertion is Proposition~\ref{prop_stable}.
\end{proof}
We call this theory the \emph{weak model companion of the pairs of models of $T$}. If the pregeometry is modular, it is the model-companion.

\begin{ex}
  The theory of pairs of any strongly minimal theory with quantifier elimination admits a weak model companion. For instance, the weak model companion of the theory of pairs of algebraically closed fields is the theory of proper pairs of algebraically closed fields and coincides with the theory of belle paires of algebraically closed fields (see~\cite{D12}, \cite{Po83}). The theory $\RCF$ also satisfies the hypotheses of Proposition~\ref{prop_weakmc}, hence the theory of pairs of real closed fields admits a weak model-companion. Connections with lovely pairs of geometric structures \cite{BV10} could be made, although we did not investigate.
\end{ex}

\subsection{A non-example: the expansion of a field of characteristic $0$ by an additive subgroup}

  \begin{prop}\label{prop_car0}
    Let $T$ be the theory of a field of characteristic $0$ in a language $\LL$ containing $\LLr$, such that $T$ is inductive. Let $\LL_G = \LL\cup\set{G}$ and let $T_G$ be the $\LL_G$-theory of models of $T$ in which $G$ is a predicate for an additive subgroup of the field.
    Let $(K,G)$ be an existentially closed model of $T_G$. Then $$S_K(G) := \set{a\in K \mid aG \subseteq G} = \Z.$$
    In particular, the theory $T_G$ does not admit a model-companion.
  \end{prop}
  \begin{proof}
    The right to left inclusion is trivial. Assume that $a\in K\setminus \Z$, let $L$ be a proper elementary extension of $K$ and $t\in L\setminus K$. Then $(L, G+\Z\frac{t}{a})$ is an $\LL_G$-extension of $(K,G)$. Furthermore, as $a\notin \Z$, we have $t\notin G+\Z\frac{t}{a}$. Then $\frac{t}{a}\in G+\Z\frac{t}{a}$ and $a\frac{t}{a} \notin G+\Z\frac{t}{a}$. As $(K,G)$ is existentially closed in $(L, G+\Z\frac{t}{a})$, we have that $$(K,G)\models \exists x (x\in G\wedge ax\notin G)$$hence $a\notin S_K(G)$. The class of existentially closed models of $T_G$ is not axiomatisable as the definable infinite set $S_L(G)$ is of fixed cardinality. As $T_G$ is inductive, this is equivalent to saying that $T_G$ does not admit a model-companion.
  \end{proof}

  \begin{rk}\label{rk_div}
    Let $T$ be the theory of a field of characteristic $0$ in a language $\LL$ containing $\LLr$, such that $T$ is inductive. Let $\LL_D = \LL\cup\set{D}$ and let $T_D$ be the $\LL_D$-theory of models of $T$ in which $D$ is a predicate for a \emph{divisible} additive subgroup of the field. Let $(K,D)$ be an existentially closed model of $T_D$.
    A similar argument yields that $\set{a\in K \mid aD = D} = \Q$, so $T_D$ does not admits a model-companion either.
  \end{rk}

\begin{rk}
  Let $K = \C$ (or $\R$). Using Remark~\ref{rk_div} and compactness arguments, one deduces that there exist $k,l\in \N$ and a constructible set if $K=\C$ (or a semialgebraic set if $K=\R$) $V\subseteq K^{k}\times K^{l}$ such that for all polynomials $P(X,Y)\in K[X,Y]$ with $\abs{X}=1$, $\abs{Y} = l$ and for all $n\in \N$ and all $q_1,\dots q_n, s_1,\dots,s_k \in \Q$ there exists $b\in K^{l}$ such that for all $a\in K^{k}$, if $(a,b)\in V$ then 
    \begin{enumerate}
      \item $a$ is not $\Q$-linearly independent over $\overline{\Q(b)}\cap K$; 
     \item $\sum_{i=1}^k s_i a_i \notin  q_1 R+\cdots + q_n R$ for $R$ the set of roots of $P(X,b)$ in $K$.
   \end{enumerate}
\end{rk}

\begin{rk}
  Recent work from Haykazyan and Kirby~\cite{HK18}, highlights a new source of $\NSOP 1$ theories, in the sense of positive logic. They study the class of existentially closed exponential fields (an exponential field is a field with a group homomorphism from the additive group to the multiplicative group of the field). 
Haykazyan and Kirby~\cite{HK18} adapted the result of Chernikov and Ramsey ~\cite{CR16} to prove that the class of existentially closed exponential fields is $\NSOP 1$ in the sense of positive logic, using the existence of a well-behaved independence relation. It is likely that the theory developed by Haykazyan and Kirby can be used to show that the class of algebraically closed fields of characteristic $0$ with a generic additive subgroup is $\NSOP 1$ in the sense of positive logic.
\end{rk}

\textbf{Acknowledgements}. This work is part of the author's Ph.D. dissertation. The author would like to thank Thomas Blossier and Zoé Chatzidakis for their countless advices and comments. Many thanks also to Nick Ramsey and Gabriel Conant for many fruitful discussions on $\NSOP{1}$ theory and examples. Thanks to Minh Chieu Tran for pointing out to me the results leading to the example of Section~\ref{sec_genmult} and many useful discussions. The author is grateful to Amador Martin-Pizarro for his professional and personal help. Finally, the author is grateful to Itay Kaplan for numerous remarks on the content.

\bibliographystyle{plain}
\bibliography{biblio}

\begin{thebibliography}{10}

\bibitem{Ad08}
Hans Adler.
\newblock Around the strong order property and mock simplicity.
\newblock {\em unpublished note}, 2008.

\bibitem{A09}
Hans Adler.
\newblock A geometric introduction to forking and thorn-forking.
\newblock {\em J. Math. Log.}, 9(1):1--20, 2009.

\bibitem{BC18}
Silvia Barbina and Enrique Casanovas.
\newblock Model theory of {S}teiner triple systems.
\newblock {\em arXiv:1805.06767 [math.LO]}, 2018.

\bibitem{BGH13}
Martin Bays, Misha Gavrilovich, and Martin Hils.
\newblock {Some Definability Results in Abstract Kummer Theory}.
\newblock {\em International Mathematics Research Notices},
  2014(14):3975--4000, 04 2013.

\bibitem{BY03}
Itay Ben-Yaacov.
\newblock Positive model theory and compact abstract theories.
\newblock {\em J. Math. Log.}, 3(1):85--118, 2003.

\bibitem{BY03b}
Itay Ben-Yaacov.
\newblock Simplicity in compact abstract theories.
\newblock {\em J. Math. Log.}, 3(2):163--191, 2003.

\bibitem{BV10}
Alexander Berenstein and Evgueni Vassiliev.
\newblock On lovely pairs of geometric structures.
\newblock {\em Ann. Pure Appl. Logic}, 161(7):866--878, 2010.

\bibitem{C99}
Zo\'e Chatzidakis.
\newblock Simplicity and independence for pseudo-algebraically closed fields.
\newblock {\em Models and Computability, S.B. Cooper, J.K. Truss Ed., London
  Math. Soc. Lect. Notes Series 259}, pages 41 -- 61, 1999.

\bibitem{C02}
Zo\'e Chatzidakis.
\newblock Properties of forking in omega-free pseudo-algebraically closed
  fields.
\newblock {\em J. Symbolic Logic}, 67(3):957--996, 2002.

\bibitem{CH99}
Zo\'{e} Chatzidakis and Ehud Hrushovski.
\newblock Model theory of difference fields.
\newblock {\em Trans. Amer. Math. Soc.}, 351(8):2997--3071, 1999.

\bibitem{CP98}
Zo\'e Chatzidakis and Anand Pillay.
\newblock Generic structures and simple theories.
\newblock {\em Ann. Pure Appl. Logic}, 95(1):71--92, 1998.

\bibitem{Cha17}
Zoé Chatzidakis.
\newblock Amalgamation of types in pseudo-algebraically closed fields and
  applications.
\newblock {\em Journal of Mathematical Logic}, 0(0):1950006, 2019.

\bibitem{ChaHru04}
Zoé Chatzidakis and Ehud Hrushovski.
\newblock Perfect pseudo-algebraically closed fields are algebraically bounded.
\newblock {\em Journal of Algebra}, 271:627--637, 01 2004.

\bibitem{Cher76}
Gregory Cherlin.
\newblock {\em Model theoretic algebra---selected topics}.
\newblock Lecture Notes in Mathematics, Vol. 521. Springer-Verlag, Berlin-New
  York, 1976.

\bibitem{CR16}
Artem Chernikov and Nicholas Ramsey.
\newblock On model-theoretic tree properties.
\newblock {\em J. Math. Log.}, 16(2):1650009, 41, 2016.

\bibitem{CKr17}
Gabriel Conant and Alex Kruckman.
\newblock Independence in generic incidence structures.
\newblock {\em 1709.09626v1 [math.LO]}, 2017.

\bibitem{dE18B}
Christian d'Elbée.
\newblock Forking, imaginaries and other features of {ACFG}.
\newblock {\em arXiv:1812.09378 [math.LO]}, 2018.

\bibitem{Del88}
Fran\c{c}oise Delon.
\newblock Id\'{e}aux et types sur les corps s\'{e}parablement clos.
\newblock {\em M\'{e}m. Soc. Math. France (N.S.)}, (33):76, 1988.

\bibitem{D12}
Fran\c{c}oise Delon.
\newblock \'{E}limination des quantificateurs dans les paires de corps
  alg\'{e}briquement clos.
\newblock {\em Confluentes Math.}, 4(2):1250003, 11, 2012.

\bibitem{DS04}
Mirna Dzamonja and Saharon Shelah.
\newblock On <*-maximality.
\newblock {\em Annals of Pure and Applied Logic}, 125(1):119–158, 2004.

\bibitem{Ga05}
Jerry Gagelman.
\newblock Stability in geometric theories.
\newblock {\em Annals of Pure and Applied Logic}, 132(2):313 -- 326, 2005.

\bibitem{G99}
Nicholas Granger.
\newblock {\em Stability, simplicity, and the model theory of bilinear forms}.
\newblock PhD thesis, University of Manchester, 1999.

\bibitem{HK18}
Levon Haykazyan and Jonathan Kirby.
\newblock Existentially closed exponential fields.
\newblock {\em arXiv:1812.08271 [math.LO]}, 2018.

\bibitem{Joh16}
Will Johnson.
\newblock {\em Fun with fields}.
\newblock PhD thesis, University of California, Berkeley, 2016.

\bibitem{KR17}
Itay Kaplan and Nicholas Ramsey.
\newblock On {K}im-independence.
\newblock {\em arXiv:1702.03894 [math.LO]}, 2017.

\bibitem{KrR18}
Alex Kruckman and Nicholas Ramsey.
\newblock Generic expansion and skolemization in {NSOP1} theories.
\newblock {\em Annals of Pure and Applied Logic}, 169(8):755 -- 774, 2018.

\bibitem{KTW18}
Alex Kruckman, Minh~Chieu Tran, and Erik Walsberg.
\newblock Interpolative fusions.
\newblock {\em arXiv:1811.06108 [math.LO]}, 2018.

\bibitem{Mac71}
Angus Macintyre.
\newblock On {$\omega _{1}$}-categorical theories of abelian groups.
\newblock {\em Fund. Math.}, 70(3):253--270, 1971.

\bibitem{MMP96}
David Marker, Margit Messmer, and Anand Pillay.
\newblock {\em Model Theory of Fields}.
\newblock Lecture Notes in Logic. Springer, 1 edition, 1996.

\bibitem{Mor96}
Patrick Morandi.
\newblock {\em Field and Galois Theory}.
\newblock Graduate Texts in Mathematics 167. Springer-Verlag New York, 1
  edition, 1996.

\bibitem{Pil00}
Anand Pillay.
\newblock Forking in the category of existentially closed structures.
\newblock In {\em Connections between model theory and algebraic and analytic
  geometry}, volume~6 of {\em Quad. Mat.}, pages 23--42. Dept. Math., Seconda
  Univ. Napoli, Caserta, 2000.

\bibitem{Po83}
Bruno Poizat.
\newblock Paires de structures stables.
\newblock {\em Journal of Symbolic Logic}, 48(2):239–249, 1983.

\bibitem{Ram18}
Nick Ramsey.
\newblock {\em Independence, Amalgamation, and Trees}.
\newblock PhD thesis, University of California, Berkeley, 2018.

\bibitem{She75}
Saharon Shelah.
\newblock The lazy model-theoretician's guide to stability.
\newblock {\em Logique et Analyse (N.S.)}, 18(71-72):241--308, 1975.
\newblock Comptes Rendus de la Semaine d'\'{E}tude en Th\'{e}orie des Mod\`eles
  (Inst. Math., Univ. Catholique Louvain, Louvain-la-Neuve, 1975).

\bibitem{SU08}
Saharon Shelah and Alexander Usvyatsov.
\newblock More on {SOP$_1$} and {SOP$_2$}.
\newblock {\em Annals of Pure and Applied Logic}, 155(1):16–31, 2008.

\bibitem{TZ12}
Katrin Tent and Martin Ziegler.
\newblock {\em A course in model theory}, volume~40 of {\em Lecture Notes in
  Logic}.
\newblock Association for Symbolic Logic, La Jolla, CA; Cambridge University
  Press, Cambridge, 2012.

\bibitem{MCT17}
Minh~Chieu Tran.
\newblock Tame structures via multiplicative character sums on varieties over
  finite fields.
\newblock {\em arXiv:1704.03853v3 [math.LO]}, 2017.

\bibitem{W75}
Peter~M Winkler.
\newblock Model-completeness and {S}kolem expansions.
\newblock {\em Model Theory and Algebra}, 1975.

\end{thebibliography}

\end{document}